\documentclass[a4paper,11pt]{article}

\usepackage[utf8]{inputenc}
\usepackage[T1]{fontenc}
\usepackage[ngerman,british,UKenglish,USenglish,american]{babel}
\usepackage{amsmath}
\usepackage{amsfonts}
\usepackage{hyperref}
\usepackage{theorem}
\usepackage{graphics}
\usepackage{makeidx}
\usepackage{fancyhdr}
\usepackage{hyperref}
\hypersetup{colorlinks=true}
\usepackage{verbatim}
\usepackage{mathrsfs}
\usepackage{bbm}
\usepackage{bm}
\usepackage{color}
\usepackage{amssymb}

\numberwithin{equation}{section}
\newcommand{\be}{\begin{equation}}
	\newcommand{\ee}{\end{equation}} 
\newcommand{\benn}{\begin{equation*}}
	\newcommand{\eenn}{\end{equation*}}
\newcommand{\bea}{\begin{eqnarray}}
	\newcommand{\eea}{\end{eqnarray}}

\newcommand{\beann}{\begin{eqnarray*}}
	\newcommand{\eeann}{\end{eqnarray*}}

\newtheorem{theorem}{Theorem}[section]

\newtheorem{corollary}[theorem]{Corollary}
\newtheorem{lemma}[theorem]{Lemma}
\newcommand{\xx}{\mathbb{X}}

\newtheorem{definition}[theorem]{Definition}
\theorembodyfont{\rm}

\newtheorem{remark}[theorem]{Remark}
\newtheorem{example}[theorem]{Example}
\newtheorem{assumptions}[theorem]{Assumptions}

\newcommand{\qed}{\hfill $\Box$\smallskip}
\newcommand{\E}{\noindent{$\mathbb{E}$ \ }}

\newcommand{\X}{\mathbf{X}}
\newcommand{\XX}{\mathbb{X}}

\def\R{\mathbb{R}}

\def\N{\mathbb{N}}

\def\P{\mathbb{P}}
\def\E{\mathbb{E}}

\def\Q{\mathbb{Q}}

\def\P{\mathbb{P}}

\def\cA{\mathcal{A}}
\def\cB{\mathcal{B}}
\def\cC{\mathcal{C}}
\def\cD{\mathcal{D}}

\def\cL{\mathcal{L}}

\def\cO{\mathcal{O}}
\def\cP{\mathcal{P}}

\def\txtd{{\textnormal{d}}}

\def\Id{{\textnormal{Id}}}

\newcommand{\norm}[1]{\left\lVert #1 \right\rVert}
\newcommand{\abs}[1]{\left| #1 \right|}

\newcommand*\samethanks[1][\value{footnote}]{\footnotemark[#1]}

\allowdisplaybreaks

\title{Existence and regularity of random attractors for stochastic evolution equations driven by rough noise}

\author{Alexandra Blessing (Neam\c tu)\thanks{University of Konstanz, Department of Mathematics and Statistics,  Universit\"atsstra\ss{}e~10 78464 Konstanz, Germany. E-Mail: alexandra.neamtu@uni-konstanz.de, tim.seitz@uni-konstanz.de}
	~~and~~Tim Seitz\samethanks
}
\begin{document}
	\maketitle
	\vspace{-1cm}
	\begin{center}
		\textit{Dedicated to the memory of Björn Schmalfuss}
	\end{center}
	\begin{abstract}
		This work establishes the existence and regularity of random pullback attractors for parabolic  partial differential equations with rough nonlinear multiplicative noise under natural assumptions on the coefficients. To this aim, we combine tools from rough path theory and random dynamical systems.~An application is given by partial differential equations with rough boundary noise, for which flow transformations are not available. 
	\end{abstract}
	
	\section{Introduction}
	We investigate the asymptotic behavior of the semilinear parabolic evolution rough partial differential equation (RPDE) given by  \begin{align}\label{eq:MainEq}
		\begin{cases}
			\txtd y_t = (Ay_t+F(y_t))~\txtd t + G(y_t)~\txtd \X_t,& \\
			y(0)=y_0\in E, &
		\end{cases}
	\end{align}
	where $E$ is a separable Banach space, $A:D(A)\subset E\to E$ is the generator of an exponentially stable analytic semigroup, $F,G$ are nonlinear terms and $\X$ is a $\gamma$-H\"older Gaussian rough path for $\gamma\in (\frac{1}{3},\frac{1}{2}]$. Under suitable assumptions on the coefficients, we prove the existence and regularity of a pullback attractor, which is a compact random set of the phase space describing the long-time behavior of~\eqref{eq:MainEq}.
	There are several major technical difficulties one encounters, when trying to investigate the existence of invariant sets for stochastic partial differential equations (SPDEs). A first conceptual difficulty is to employ the concept of random dynamical systems (RDS)~\cite{Arnold} for SPDEs. It is well-known that an It\^{o}-type SDE generates an RDS under natural assumptions on the coefficients~\cite{Arnold,Kunita90}. However, the generation of an RDS from an It\^{o}-type SPDE has been a long-standing open problem, mostly since Kolmogorov's theorem breaks down for random fields parametrized by infinite-dimensional Hilbert spaces. As a consequence, it is not trivial, to obtain a RDS from a general SPDE. This problem was fully solved only under very restrictive assumptions on the structure of the noise driving the SPDE.~For instance, if one deals with purely additive noise or linear multiplicative Stratonovich noise, there are standard flow transformations which reduce the SPDE to a PDE with random non-autonomous coefficients.~Since this random PDE can be solved pathwise, it is straightforward to obtain an RDS. However, for nonlinear multiplicative noise, this technique is no longer applicable.~In the rough path setting, these issues do not occur, due to the fact that the solution exists pathwise, without transforming the SPDE into a random PDE. Provided that the noise forms a rough path cocycle, the generation of a random dynamical system from rough differential equations was established in~\cite{BRiedelScheutzow}. For RPDEs this aspect was considered in \cite{Hofmanova2,HN20,HN22, KN23}, where a major difficulty is the global existence of the solution.\\

	There are numerous results on the existence of attractors for stochastic partial differential equations, starting with the famous works by Crauel and Flandoli~\cite{CrauelFlandoli} and Schmalfuss~\cite{FlandoliSchmalfuss96,Schmalfuss}. These are applicable to equations driven by additive~\cite{Gess} or linear multiplicative noise~\cite{Gess:m} and use  flow transformations.~However, there are very few works~\cite{CaoGao23,GaoMS2014,LinYangZeng23} which use a pathwise solution theory to study attractors for SPDEs. Moreover, all these techniques lead to certain restrictions on the diffusion coefficient $G$~\cite{LinYangZeng23} or on the noise~\cite{GaoMS2014}. In~\cite{LinYangZeng23} the growth bound of the Lipschitz continuous nonlinear term $G$ is assumed to be smaller than one, whereas in~\cite{GaoMS2014} the trace of the infinite-dimensional fractional Brownian motion is assumed to be small. The techniques in~\cite{CaoGao23,GaoMS2014, LinYangZeng23} rely on stopping times which directly control the size of the noise on a certain time interval and are required to be on average smaller than one, leading to the restrictions mentioned above.\\
	
	Our approach is based on the integrable solution bounds for rough partial differential equations recently obtained in~\cite{RiedelVarzaneh23}. This technique uses Greedy times~\cite{Cass,FrizRiedel} instead of stopping times, which are defined in terms of a control, see Definition~\ref{def:greedyTimes}. Their number $N_{s,t}(\omega)$ on a time interval $[s,t]$ has much better integrability properties than the $p$-variation (or H\"older) norm of the given rough path. Modifying the sewing lemma in order to incorporate the control instead of the H\"older norm of the solution,~\cite{RiedelVarzaneh23} derived an integrable bound on the solution of~\eqref{eq:MainEq} on a time interval $[s,t]$ in terms of $N_{s,t}(\omega)$ and a polynomial expression containing the H\"older norms of the driving rough path. 
	Such bounds already turned out to be useful for the stability and for the existence of invariant manifolds for rough differential equations~\cite{RV23} improving the assumptions on the coefficients made in~\cite{KN21}. In our case, this method improves the restriction on the smallness of the growth bound of the coefficient $G$ required for the existence of the attractor in~\cite{LinYangZeng23}.~Another major advantage of this method is that it directly allows us to investigate the regularity of the random attractor. This is natural for parabolic SPDEs perturbed by additive or linear multiplicative noise. For instance, in~\cite[Section 3.1]{D97} a stochastic reaction-diffusion equation with Dirichlet boundary conditions on a domain $\cO$, driven by finite-dimensional additive noise was considered on the phase space $L^2(\cO)$. It was shown that its random attractor is contained in a ball of $H^1_0(\cO)$ with random radius. To the best of our knowledge, there are no results on the regularity of attractors for rough PDEs as~\eqref{eq:MainEq} considered here. All the available results use flow transformations~\cite{D97,Bao} and the references specified therein. We also believe that these tools can be further used to obtain a bound on the box dimension of the attractor or to prove the existence of exponential attractors~\cite{KuhnNeamtuSonner21} which automatically entails a bound on their box dimension. 
	We leave these aspects for a future work. \\
	
	This manuscript is structured as follows. Section~\ref{sec:p} collects basic results from rough path theory and random dynamical systems. Section~\ref{sec:main} contains the main results (Lemma~\ref{lem:AbsorbingSet}, Theorem~\ref{thm:attractor}, Corollary~\ref{cor:reg})  on the existence and regularity of a random pullback attractor for~\eqref{eq:MainEq}. Analogously to~\cite{KN23,LinYangZeng23,KN21,RV23} we discretize~\eqref{eq:MainEq} and use the integrable bounds derived in~\cite{RiedelVarzaneh23} to estimate the controlled rough path norm of the solution on a  time interval of length one.
	Thereafter, we establish the existence of a random absorbing set for~\eqref{eq:MainEq} using a version of the singular Henry-Gronwall lemma~\cite{Henry81}. This is necessary due to the fact that the nonlinear term $F$ is allowed to lose spatial regularity and leads to an a-priori estimate which contains the derivative of the Mittag-Leffler function. This is a generalization of the exponential function and therefore has a similar asymptotic behavior~\cite{Henry81,SellYou}. This fact is crucial for our aims, since we are interested in the long-time behavior of~\eqref{eq:MainEq}. The arguments in the computation of the absorbing set based on the singular Gronwall lemma and rough path estimates are new 
	and of independent interest. 
	
	The existence of an absorbing set together with a compactness argument for parabolic PDEs yield the existence of an attractor for~\eqref{eq:MainEq}.
	In Section~\ref{sec:app} we provide examples for the coefficients of~\eqref{eq:MainEq} which satisfy the conditions imposed for the absorbing set in Section~\ref{sec:main}.~A main application which fits into our framework is given by parabolic PDEs with rough boundary noise. The well-posedness and generation of a random dynamical system in this case was investigated  in~\cite{NS23}. In this work we go a step further establish the existence of random pullback attractors. To our best knowledge, this is the first result in this direction for this type of equations.~Even if the noise is additive or linear multiplicative, since it is acting on the boundary, it is not possible to transform such equations into PDEs with random coefficients. 
	
	\section*{Ackwnowledgements} We are grateful to Mazyar Ghani Varzaneh for very helpful discussions regarding~\cite{RiedelVarzaneh23}. 
	
	\section{Preliminaries}\label{sec:p}
	In this section, we collect some basic results on rough paths \cite{FrizHairer,FrizVictoir10} and rough convolutions for semilinear parabolic problems~\cite{GHN2021} as well as concepts from the theory of random dynamical systems. Furthermore, we provide the assumptions on the coefficients of~\eqref{eq:MainEq} and the construction of a sequence of greedy time steps required in order to establish integrable bounds for the rough integral \cite{RiedelVarzaneh23}. 
	\subsection{Rough path theory}
	We denote by $C^\alpha([0,T];E)$ the $\alpha$-H\"older continuous paths with values in some Banach space $E$, and write $\left[\cdot\right]_{\alpha,E}$ for the H\"older seminorm as well as $\norm{\cdot}_{\infty,E}$ for the supremums norm. In the case of $E:=\R$, we write $\left[\cdot\right]_{\alpha}$ or $\left[\cdot\right]_{\alpha,[0,T]}$ to emphasize the time dependence. When the time interval is clear from the context, we  use the abbreviation $C^{\gamma}(E)$ to point out the interplay between space and time regularity. Furthermore, let be $\Delta_J:=\{ (s,t)\in J \times J : s\leq t \}$ for a compact interval $J\subset \R$. Then, the noise is a $d$-dimensional $\gamma$-H\"older rough path $\textbf{X}:=(X,\mathbb{X})$, for $\gamma\in(1/3,1/2]$ with $X_0=0$. Here we assume without loss of generality that $d=1$, since the generalization to $d>1$ can be made componentwise.
	More precisely, we have
	\begin{align*}
		X\in C^{\gamma}([0,T];\mathbb{R}) ~~\mbox{ and } ~~ \mathbb{X}\in C^{2\gamma}(\Delta_{[0,T]};\mathbb{R}\otimes\mathbb{R})
	\end{align*}
	and the connection between $X$ and $\mathbb{X}$ is given by Chen's relation
	\begin{align}\label{chen}
		\mathbb{X}_{s,t}- \mathbb{X}_{s,u}-\mathbb{X}_{u,t}=X_{s,u}\otimes X_{u,t},
	\end{align}
	for $s\leq u\leq t$, where we write $X_{s,u}:=X_u-X_s$ for any path. Let us further introduce an appropriate metric on the space of rough paths.
	\begin{definition}
		Let $\mathbf{X}=(X,\mathbb{X})$ and $\mathbf{\widetilde{X}}=(\widetilde{X},\widetilde{\mathbb{X}})$ be two $\gamma$-H\"older rough paths and $T>0$. We introduce the $\gamma$-H\"older rough path (inhomogeneous) metric
		\begin{align}\label{rp:metric}
			d_{\gamma,J}(\mathbf{X},\mathbf{\widetilde{X}})
			:= \sup\limits_{(s,t)\in \Delta_{J}} \frac{|X_{s,t}-\widetilde{X}_{s,t}|}{|t-s|^{\gamma}}
			+ \sup\limits_{(s,t) \in \Delta_{J}}
			\frac{|\mathbb{X}_{s,t}-\widetilde{\mathbb{X}}_{s,t}|} {|t-s|^{2\gamma}},
		\end{align}
		and set $\rho_{\gamma,J}(\mathbf{X}):=d_{\gamma,J}(\mathbf{X},0)$.
	\end{definition}
	A subclass of H\"older rough paths, are those resulting from the closure of smooth functions with respect to the introduced rough path metric.
	\begin{definition}\label{def:geomRP}
		We call $\X=(X,\XX)$ a geometric $\gamma$-H\"older rough path if there exists a sequence $(X^{n})_{n\in \N}\subset C^\infty ([0,T];\R)$ such that $\X^{n}:=(X^{n},\XX^{n})$, where $\XX^{n}$ is the canonical lift of $X^{n}$
		\begin{align}\label{eq:LevyArea}
			\XX^n_{s,t}:=\int_{s}^t X^n_{s,r}\otimes\txtd X^n_r,
		\end{align}
		converges in the rough path topology to $\X$, i.e. $d_{\gamma,[0,T]}(\X^n,\X)\to 0$ for $n\to \infty$. 
	\end{definition}
	\begin{remark}\label{rem:RPContHolderNorm}
		The convergence $d_{\gamma,[0,T]}(\X_n,\X)\to 0$ for $n\to \infty$ in the definition above is equivalent to
		\begin{align*}
			\left[X_n-X\right]_{\gamma,[0,T]}\to 0\qquad \text{and}\qquad \left[\XX_n-\XX\right]_{2\gamma,\Delta_{[0,T]}}\to 0,
		\end{align*}
		for $n\to \infty$. But this means that $X\in C^{0,\gamma}([0,T];\R)$ and $\XX\in C^{0,2\gamma}(\Delta_{[0,T]};\R\otimes\R)$, where $C^{0,\gamma}$ is the closure of smooth paths with respect to the H\"older norm. Therefore, we have the continuity of $\Delta_{[0,T]}\to \R,(s,t)\mapsto \left[X\right]_{\gamma,[s,t]}$ as well as for the H\"older norm of $\XX$, 
		see for example \cite[Theorem 5.33]{FrizVictoir10}.
	\end{remark}
	Let us now specify the necessary assumptions on the linear part of~\eqref{eq:MainEq}.
	\begin{assumptions}\label{ass}
		\begin{itemize}
			\item[1)] The operator $A:D(A)\subset E \to E$ is densely defined, generates a compact analytic semigroup $(S_t)_{t\in [0,\infty)}$ and has bounded imaginary powers, see Section~\ref{sec:app} for concrete examples. 
			\item[2)] The semigroup is exponentially stable, which means that there exist constants $\tilde{\lambda}_A,C_A>0$ such that $\norm{S_t}_{\mathcal{L}(E)}\leq C_A e^{-\tilde{\lambda}_A t}$ for every $t\geq 0$.
		\end{itemize}
	\end{assumptions}
	These conditions ensure that the fractional powers of $A$ exist. So set $E_\alpha:=D(A^\alpha)$ endowed with the norm $\norm{\cdot}_\alpha:=\norm{A^\alpha \cdot}_E$ for $\alpha>0$ and also $E_{-\alpha}$ as the closure of $E$ with respect to the norm $\norm{\cdot}_{-\alpha}=\norm{A^{-\alpha}\cdot}_{E}$. We obtain a family of separable Banach spaces $(E_\alpha,\norm{\cdot}_\alpha)_{\alpha\in\mathbb{R}}$ which are continuously embedded, i.e.~$E_{\alpha}\hookrightarrow E_\beta$ for $\alpha\leq \beta$ and satisfy  
	\begin{align}\label{eq:IPolSpaces}
		E_{(\beta-\alpha)\theta+\alpha}=[E_\alpha,E_\beta]_\theta
	\end{align}
	for $\alpha\leq \beta$ and $\theta\in (0,1)$, where $[\cdot,\cdot]_\theta$ denotes the complex interpolation, see for example \cite[Theorem V.1.5.4]{Amann95}. In particular, this implies the interpolation inequality
	\begin{align}\label{interpolation:ineq}
		\norm{x}^{\alpha_3-\alpha_1}_{\alpha_2} \lesssim \norm{x}^{\alpha_3-\alpha_2}_{\alpha_1} \norm{x}^{\alpha_2-\alpha_1}_{\alpha_3},
	\end{align}
	for $\alpha_1\leq \alpha_2\leq \alpha_3$ and $x\in  E_{\alpha_3}$. Therefore $(E_\alpha,\norm{\cdot}_\alpha)_{\alpha\in\mathbb{R}}$ is an example of a monotone family of interpolation spaces as introduced in \cite[Definition 2.1]{GHN2021}. Further, the compactness of the semigroup ensures that the embeddings $E_\beta\hookrightarrow E_\alpha $ for $\beta>\alpha$ are compact.
	
	The main advantage of this approach is that we can view the semigroup $(S_t)_{t\in [0,\infty)}$ generated by $A$ as a linear mapping on these interpolation spaces and obtain a trade-off between spatial and time regularity. In particular, regarding the assumed exponential stability, we have that $S_t\in\cL(E_{\alpha})$ and for every $\lambda_A\in (0,\tilde{\lambda}_A),\alpha\in \R$, $\sigma\in [0,1]$ there exists a constant $C_{-\sigma}(\lambda_A)=C_{-\sigma}>0$ such that
	\begin{align}
		\norm{S_tx}_{\alpha+\sigma}\leq C_{-\sigma} e^{-\lambda_A t}t^{-\sigma}\norm{x}_{\alpha}\label{hg:1},
	\end{align}
	for $t\geq 0$, see~\cite[Theorem 7.7.2 iii)]{Vrabie03}. 
	Throughout the manuscript we fix $\lambda_A<\tilde{\lambda}_A$. \\
	
	Keeping this in mind, we now introduce the definition of a controlled rough path tailored to the parabolic structure of the PDE we consider, in the spirit of~\cite{GHN2021}. This is convenient for our aims, since the semigroup will not be incorporated in the definition of the controlled rough path as in~\cite{GHairer} or alternative approaches~\cite{GubinelliTindel, HN19} which iterate the stochastic convolution into itself. 
	
	\begin{definition}\label{def:crp}
		We call a pair $(y,y')$ a controlled rough path for some fixed $\alpha\in \R$ if $(y,y')\in C(E_\alpha) \times (C(E_{\alpha-\gamma} ) \cap C^{\gamma}(E_{\alpha-2\gamma}))$ and the remainder  
		\begin{align}\label{remainder}
			(s,t)\in \Delta_{[0,T]}\mapsto R^y_{s,t}:= y_{s,t} -y'_s X_{s,t}    
		\end{align}
		belongs to $ C^{\gamma}(E_{\alpha-\gamma})\cap C^{2\gamma}(E_{\alpha-2\gamma})$. The component $y'$ is often referred to as Gubinelli derivative of $y$.
		The space of controlled rough paths is denoted by $\cD^{2\gamma}_{X,\alpha}$ and endowed with the norm $\|\cdot\|_{\cD^{2\gamma}_{X,\alpha}}$ given by
		\begin{align}\label{g:norm}
			\norm{y,y'}_{\cD^{2\gamma}_{X,\alpha}}:= \left\|y \right\|_{\infty,E_\alpha} 
			+ \|y' \|_{\infty,E_{\alpha-\gamma}}
			+ \left[y'\right]_{\gamma,E_{\alpha-2\gamma}}
			+\left[R^y \right]_{\gamma,E_{\alpha-\gamma}}    + \left[R^y \right]_{2\gamma,E_{\alpha-2\gamma}}.
		\end{align}
	\end{definition}
	The first index in the notation above always indicates the time regularity, and the second one stands for the space regularity. For simplicity, we often write $\|y\|_{\infty,\alpha}:=\|y\|_{\infty,E_\alpha}$ and $[y']_{\gamma,\alpha-2\gamma}:=[y']_{\gamma,E_{\alpha-2\gamma}}$ and analogously for the remainder.
	In order to emphasize the time horizon, we write $\cD^{2\gamma}_{X,\alpha}([0,T])$ instead of $\cD^{2\gamma}_{X,\alpha}$.   
	Given this setting of controlled rough paths, one can introduce the rough convolution as follows.
	\begin{theorem}{\em (\cite[Theorem 4.5]{GHN2021}).}
		\label{integral} Let $(y,y')\in \cD^{2\gamma}_{X,\alpha}([s,t])$, then the limit
		\begin{align}\label{Gintegral}
			\int\limits_{s}^{t} S_{t-r}y_{r}~\txtd \textbf{X}_{r} :=\lim\limits_{|\mathcal{P}|\to 0} \sum\limits_{[u,v]\in\mathcal{P}} S_{t-u}y_{u}X_{v,u} + S_{t-u}y'_{u}\mathbb{X}_{v,u},
		\end{align}
		exists as an element in $E_{\alpha-2\gamma}$, where $\mathcal{P}$ denotes a partition of $[s,t]$. For $0\leq \beta< 3\gamma$ the following estimate
		\begin{align}\label{estimate:integral}
			\norm{ \int\limits_{s}^{t} S_{t-r} y_{r}~\txtd\textbf{X}_{r}}_{\alpha-2\gamma+\beta} \leq C_I \rho_{\gamma,[s,t]}(\X) \norm{y,y'}_{\cD^{2\gamma}_{X,\alpha}([s,t])} (t-s)^{3\gamma-\beta}
		\end{align}
		holds true, where $C_I:=C_I(\alpha,\gamma,\beta)>0$. 
	\end{theorem}
	We now specify which type of nonlinearities we consider.
	\begin{assumptions}\label{ass:Nonlin}
		\begin{itemize}
			\item[1)] There exists a constant $\sigma_F\in[0,1)$ such that the drift term $F:E_{\alpha}\to E_{\alpha-\sigma_F}$ is Lipschitz continuous, with constant $\widetilde{C}_F>0$, which in particular implies a linear growth condition $\norm{F(x)}_{\alpha-\sigma_F}\leq \norm{F(0)}_{\alpha-\sigma_F}+\widetilde{C}_F\norm{x}_{\alpha}$. We further set $C_F:=\max\{\norm{F(0)}_{\alpha-\sigma_F},\widetilde{C}_F\}$.
			\item[2)] There exists a $\sigma_G\in[0,\gamma)$ such that for any $\vartheta \in [0, 2\gamma]$ the diffusion term $G:E_{\alpha-\vartheta}\to E_{\alpha-\vartheta-\sigma_G}$ is bounded and three times Fr\'{e}chet differentiable with bounded derivatives, which means that $\norm{D^i G}_{\mathcal{L}(E_{\alpha-\vartheta}^{\otimes i};E_{\alpha-\vartheta-\sigma_G})}<\infty$ for $i=1,2,3$. 
			We further set $C_G:=\sup\limits_{\vartheta\in [0,2\gamma]}\max\limits_{i=1,2,3} \norm{D^i G}_{\mathcal{L}(E_{\alpha-\vartheta}^{\otimes i};E_{\alpha-\vartheta-\sigma_G})}$ and suppose that $C_G<\infty$.
			
		\end{itemize}
	\end{assumptions}
	Under these conditions, it is known that \eqref{eq:MainEq} has for every $y_0\in E_\alpha$ a unique global mild solution which is a controlled rough path $(y,y^\prime)\in \cD^{2\gamma}_{X,\alpha}([0,T])$ such that 
	\begin{align}\label{eq:mildSolution}
		y_t=S_t y_0 +\int_0^t S_{t-r}F(y_r)~\txtd r +\int_0^t S_{t-r}G(y_r)~\txtd \X_r,
	\end{align}
	for $t\geq 0$ and $y^\prime=G(y)$, see \cite[Theorem 3.8]{HN22}. One major ingredient was there, that the Gubinelli derivative satisfies $y^\prime=G(y)$. Using this identity one can establish an estimate without quadratic terms as in \cite[Lemma 4.7]{GHN2021}.
	\begin{lemma}{\em (\cite[Lemma 3.6]{HN22}).} 
		Let Assumption \ref{ass:Nonlin} be satisfied and $(y,G(y))\in \cD^{2\gamma}_{X,\alpha}([s,t])$ be a controlled rough path. Then we have $(G(y),(G(y))^\prime)\in \cD^{2\gamma}_{X,\alpha-\sigma_G}([s,t])$ with $(G(y_t))^\prime=DG(y_t)\circ G(y_t)$ and the estimate 
		\begin{align}\label{ineq:Nonlin}
			\norm{G(y),(G(y))^\prime}_{\cD^{2\gamma}_{X,\alpha-\sigma_G}([s,t])}\leq C_G \rho_{\gamma,[s,t]}(\X)(1+\norm{y,G(y)}_{\cD^{2\gamma}_{X,\alpha}([s,t])}).
		\end{align}
	\end{lemma}
	So far we assumed that $\X$ is a deterministic path. Since, the main objective is to prove the existence of a random attractor for~\eqref{eq:MainEq}, we work from now on in a stochastic setting. Therefore, we fix a probability space $(\Omega,\mathcal{F},\P)$ and recall the notion of a metric dynamical system and rough path cocycles.
	\begin{definition}\label{def:MDS}
		The quadrupel $(\Omega,\mathcal{F},\P,(\theta_t)_{t\in \R})$, where $\theta_t:\Omega\to\Omega$ is measure-preserving, is called a metric dynamical system if
		\begin{itemize}
			\itemsep -4pt
			\item[i)] $\theta_0=\Id_\Omega$,
			\item[ii)] $(t,\omega)\mapsto \theta_t\omega$ is $\mathcal{B}(\R)\otimes \mathcal{F}-\mathcal{F}$ measurable,
			\item[iii)] $\theta_{t+s}=\theta_t\circ \theta_s$ for all $t,s \in \R$. 
		\end{itemize}
		We call it an ergodic metric dynamical system if for any $A\in \mathcal{F}$, which is $(\theta_t)_{t\in \R}$ invariant, we have $\P(A)\in \{0,1\}$.
	\end{definition}
	\begin{definition}{\em (\cite[Definition 2]{BRiedelScheutzow}).}\label{def:RPCocycle}
		We call a pair
		\begin{align*}
			\mathbf{X}=(X,\xx):\Omega\to C^{\gamma}_{{\rm loc}}(\R;\mathbb{R}) \times C^{2\gamma}_{{\rm loc}}(\Delta_\R;\R\otimes \R) 
		\end{align*}
		a ($\gamma$-H\"older) rough path cocycle if $\mathbf{X}|_{[0,T]}(\omega)$ is a $\gamma$-H\"older rough path for every $T>0$ and $\omega\in\Omega$ and the cocycle property $X_{s,s+t}(\omega)= X_t(\theta_s\omega)$ as well as $\xx_{s,s+t}(\omega)=\xx_{t,0}(\theta_s\omega)$ holds true for every $s\in \R,t\in[0,\infty)$ and $\omega\in\Omega$.
	\end{definition}
	In order to study the asymptotic behavior of~\eqref{eq:MainEq}, we have to estimate the norm of the solution $y$ on large time intervals, where the main challenge is to control the $\gamma$-H\"older norm of the noise. For this purpose, we introduce a sequence of greedy time steps according to \cite{RiedelVarzaneh23} and a continuous control.~The number of Greedy time steps within an interval has better integrability properties than the H\"older norms of the noise. 
	
	
	\begin{definition}\label{def:greedyTimes}
		Let $\eta\in [0,\gamma)$, $\chi>0$, $I=[a,b]\subset\R$,  and $\X(\omega):=(X(\omega),\xx(\omega))$, $\omega\in \Omega$, be a $\gamma$-H\"older rough path cocycle.
		\begin{itemize}
			\item[i)] We define for $a\leq s\leq t\leq b$
			\begin{align*}
				W_{s,t}(\omega)&:=W_{\X(\omega),\gamma,\eta}(s,t)\\
				&:=\sup\limits_{\cP\subset [s,t]} \left\{\sum_{j=0}^{\abs{\cP}} (\kappa_{j+1}-\kappa_j)^{-\frac{\eta}{\gamma-\eta}}\left(\abs{X(\omega)_{\kappa_j,\kappa_{j+1}}}^{\frac{1}{\gamma-\eta}}+\abs{\xx(\omega)_{\kappa_j,\kappa_{j+1}}}^{\frac{1}{2(\gamma-\eta)}}\right) \right\},
			\end{align*}
			where the supremum is taken over all partitions $\cP=\{s=\kappa_0<\kappa_1<\ldots<\kappa_n=t\}$ of $[s,t]$.
			\item[ii)] We define the sequence of greedy times $(\tau_n(\omega))_{n\in\N_0}$ through $\tau_0(\omega):=a$ and recursively for $n\in \N_0$
			\begin{align*}
				\tau_{n+1}(\omega):=\tau_{n+1,\eta,\omega}^I(\chi):=\sup \left\{\tau\in [\tau_{n}(\omega),b]~|~ W^{\gamma-\eta}_{\tau_{n}(\omega),\tau}(\omega)\leq \chi\right\}.
			\end{align*}
			\item[iii)] Let  
			\begin{align*}
				N_{a,b}(\omega):=N(I,\eta,\chi,\X(\omega)):=\inf\{n> 0~|~\tau_n^I(\omega)=b\},
			\end{align*}
			be the number of greedy time steps in the interval $I$.
		\end{itemize}
	\end{definition}
	For a better readability, we omit in the following the  dependencies of $(\tau_n(\omega))_{n\in \N_0},N_{a,b}(\omega)$ and $W_{s,t}(\omega)$ on $\gamma,\eta, \chi$ and $I$, whenever this dependence is clear from the context. \\
	
	The following properties are direct consequences of the previous definition. 
	\begin{lemma}\label{lem:GreedTime} 
		Let $\eta\in [0,\gamma)$, $\chi>0$, $I=[a,b]\subset\R,~ \omega\in \Omega$ and $\X:=(X,\xx)$ be a $\gamma$-Hölder rough path cocycle.
		\begin{itemize} 
			\item[i)] $W$ is a continuous control, this means that for $s\leq u\leq t$ the following subadditive property
			is satisfied \begin{align*}
				W_{s,u}(\omega)+W_{u,t}(\omega)\leq W_{s,t}(\omega).
			\end{align*}
			\item[ii)] For $s\leq t$ we have the bounds
			\begin{align*}
				N_{s,t}(\omega)&\leq W_{s,t}(\omega)\chi^{-\frac{1}{\gamma-\eta}}+1,\\
				W_{s,t}(\omega)&\leq (t-s)\left(\left[X(\omega)\right]_{\gamma,[s,t]}^{\frac{1}{\gamma-\eta}}+\left[\xx(\omega)\right]^{\frac{1}{2(\gamma-\eta)}}_{2\gamma,\Delta_{[s,t]}}\right).
			\end{align*}
			\item[iii)] We have $\left[X(\theta_r \omega)\right]_{\gamma,[s,t]}=\left[X(\omega)\right]_{\gamma,[s+r,t+r]}$ and $\left[\XX(\theta_r \omega)\right]_{2\gamma,\Delta_{[s,t]}}=\left[\XX(\omega)\right]_{2\gamma,\Delta_{[s+r,t+r]}}$ for $r\in \R$. In particular, the same holds for $W_{s,t}(\theta_{r}\omega)=W_{s+r,t+r}(\omega)$ and $N_{s,t}(\theta_{r}\omega)=N_{s+r,t+r}(\omega)$. 
		\end{itemize}
	\end{lemma}
	\begin{proof}
		\begin{itemize}
			\item[i)] This is straightforward. 
			\item[ii)] Due to the subadditivity of $W_{s,t}(\omega)$ we get
			\begin{align*}
				N_{s,t}(\omega)-1\leq \sum_{j=0}^{N_{s,t}(\omega)-2} W_{\tau_j,\tau_{j+1}}(\omega)\chi^{-\frac{1}{\gamma-\eta}}\leq W_{s,\tau_{N_{s,t}(\omega)-1}}(\omega)\chi^{-\frac{1}{\gamma-\eta}}\leq W_{s,t}(\omega)\chi^{-\frac{1}{\gamma-\eta}},
			\end{align*}
			which leads to the first estimate. For the second inequality, we note that
			\begin{align*}
				&\sum_{j=0}^{\abs{\cP}} (\kappa_{j+1}-\kappa_j)^{-\frac{\eta}{\gamma-\eta}}\left(\abs{X(\omega)_{\kappa_j,\kappa_{j+1}}}^{\frac{1}{\gamma-\eta}}+\abs{\xx(\omega)_{\kappa_j,\kappa_{j+1}}}^{\frac{1}{2(\gamma-\eta)}}\right) \\
				&=\sum_{j=0}^{\abs{\cP}} (\kappa_{j+1}-\kappa_j)^{1-\frac{\gamma}{\gamma-\eta}}\left(\abs{X(\omega)_{\kappa_j,\kappa_{j+1}}}^{\frac{1}{\gamma-\eta}}+\abs{\xx(\omega)_{\kappa_j,\kappa_{j+1}}}^{\frac{1}{2(\gamma-\eta)}}\right) \\
				&\leq \sum_{j=0}^{\abs{\cP}} (\kappa_{j+1}-\kappa_j)\left(\left(\frac{\abs{X(\omega)_{\kappa_j,\kappa_{j+1}}}}{(\kappa_{j+1}-\kappa_j)^\gamma}\right)^{\frac{1}{\gamma-\eta}}+\left(\frac{\abs{\xx(\omega)_{\kappa_j,\kappa_{j+1}}}}{(\kappa_{j+1}-\kappa_j)^{2\gamma}}\right)^{\frac{1}{2(\gamma-\eta)}}\right) \\
				&\leq \sum_{j=0}^{\abs{\cP}} (\kappa_{j+1}-\kappa_j)\left(\left[X(\omega)\right]_{\gamma,[\kappa_j,\kappa_{j+1}]}^{\frac{1}{\gamma-\eta}}+\left[\xx(\omega)\right]_{2\gamma,\Delta_{[\kappa_j,\kappa_{j+1}]}}^{\frac{1}{2(\gamma-\eta)}}\right) \\
				&\leq (t-s)\left(\left[X(\omega)\right]_{\gamma,[s,t]}^{\frac{1}{\gamma-\eta}}+\left[\xx(\omega)\right]_{2\gamma,\Delta_{[s,t]}}^{\frac{1}{2(\gamma-\eta)}}\right),
			\end{align*}
			yielding the second estimate, since the right-hand side is independent of the choice of the partition.
			\item[iii)] Follows directly due to the cocycle property of the rough path cocycle $\X(\omega)$.
		\end{itemize}
		\qed
	\end{proof}
	\begin{theorem}{\em (\cite[Theorem 2.13]{RiedelVarzaneh23}).}\label{thm:VarRiedelEst} There exists a constant $\widetilde{M}>1$ and $\chi \in (0,1)$, $\eta\in (\sigma_G,\gamma)$ with $e^{\widetilde{M}}\chi^{\gamma-\eta}\leq \frac{1}{2}$ such that the following statements are true. 
		\begin{itemize}
			\item[i)] For $s\leq t$ such that $t-s\leq 1$, $t-s\leq (4\widetilde{M})^{-\frac{1}{1-\max\{\sigma_F,2\gamma\}}}=:d$ and any fixed $\omega\in \Omega$ the solution $(y,y^\prime)$ of \eqref{eq:MainEq} satisfies the inequality
			\begin{align}\label{ineq:VarRiedelEst1}
				\begin{split}
					&\sup\limits_{r\in[s,t]} \norm{y_r}_{\alpha}\leq \sup\limits_{0\leq n \leq N([s,t],\eta,\chi,\X(\omega))-1}\norm{y,y^\prime}_{\cD^{2\gamma}_{X(\omega),\alpha}([\tau_n,\tau_{n+1}])} \\
					&\leq \left(e^{N([s,t],\eta,\chi,\X(\omega))\widetilde{M}}\norm{y_s}_\alpha+\frac{e^{N([s,t],\eta,\chi,\X(\omega))\widetilde{M}+\widetilde{M}}-1}{e^{\widetilde{M}}-1}P(\left[X(\omega)\right]_{\gamma,[s,t]},\left[\xx(\omega)\right]_{2\gamma,\Delta_{[s,t]}})\right),
				\end{split}
			\end{align}
			where $P(x,y):=1+x+y+x(x^2+y).$
			\item[ii)] For arbitrary intervals with $1\geq t-s>d$ and any fixed $\omega\in \Omega$, the solution $(y,y^\prime)$ of \eqref{eq:MainEq} satisfies 
			for a deterministic constant $M>0$ and $\widetilde{N}:=\lceil d^{-1}(t-s)\rceil$ the estimate \begin{align}\label{ineq:VarRiedelEst}
				\begin{split}
					&\norm{y,y^\prime}_{\cD^{2\gamma}_{X(\omega),\alpha}([s,t])}\leq \widetilde{N}\left(M N_{s,t}(\omega)(1+\left[X(\omega)\right]_{\gamma,[s,t]})e^{N_{s,t}(\omega)\widetilde{M}}\right)^{\widetilde{N}+1}\norm{y_{s}}_\alpha\\
					&+M \widetilde{N}N_{s,t}(\omega)(1+\left[X(\omega)\right]_{\gamma,[s,t]})\frac{e^{N_{s,t}(\omega)\widetilde{M}+\widetilde{M}}-1}{e^{\widetilde{M}}-1}P(\left[X(\omega)\right]_{\gamma,[s,t]},\left[\xx(\omega)\right]_{2\gamma,\Delta_{[s,t]}})\\
					&\times\sum_{k=1}^{\widetilde{N}}  \left(M N_{s,t}(\omega)(1+\left[X(\omega)\right]_{\gamma,[s,t]})e^{N_{s,t}(\omega)\widetilde{M}}\right)^k.
				\end{split}
			\end{align}
			
		\end{itemize}
	\end{theorem}
	\begin{proof}
		The inequality \eqref{ineq:VarRiedelEst} is proved in \cite[Theorem 2.13 i)]{RiedelVarzaneh23}. For the second inequality, note that for a partition $(t_k)_{k=0}^{\widetilde{N}}$ of $[s,t]$ there exists a constant $M>0$ such that
		\begin{align}\label{est:DiscreteGubNorm}
			\norm{y,y^\prime}_{\cD^{2\gamma}_{X(\omega),\alpha}([s,t])}\leq M(1+[X(\omega)]_{\gamma,[s,t]}) \sum_{k=0}^{\widetilde{N}-1} \norm{y,y^\prime}_{\cD^{2\gamma}_{X(\omega),\alpha}([t_k,t_{k+1}])}
		\end{align}
		as well as
		\begin{align}\label{ineq:VarRiedelEst2}
			\norm{y,y^\prime}_{\cD^{2\gamma}_{X(\omega),\alpha}([s,t])} \leq p_1 (\omega,[s,t])\norm{y_s}_\alpha+p_2(\omega,[s,t]),
		\end{align}
		where $t-s\leq d$ and
		\begin{align*}
			\begin{split}
				p_1(\omega,[s,t])&:=M N_{s,t}(\omega)(1+\left[X(\omega)\right]_{\gamma,[s,t]})e^{N_{s,t}(\omega)\widetilde{M}},\\
				p_2(\omega,[s,t])&:=M N_{s,t}(\omega)(1+\left[X(\omega)\right]_{\gamma,[s,t]})\frac{e^{N_{s,t}(\omega)\widetilde{M}+\widetilde{M}}-1}{e^{\widetilde{M}}-1}P(\left[X(\omega)\right]_{\gamma,[s,t]},\left[\xx(\omega)\right]_{2\gamma,\Delta_{[s,t]}}).
			\end{split}
		\end{align*}
		To extend the statement for arbitrary $t-s>d$, we define iteratively $t_0=s$ and $t_k:=\min\{d+t_{k-1},t\}$ for $k=1,\ldots,\widetilde{N}=\lceil d^{-1}(t-s)\rceil$. Then $t_{k+1}-t_k\leq d$, and we can use \eqref{ineq:VarRiedelEst2} to obtain 
		\begin{align*}
			&\norm{y,y^\prime}_{\cD^{2\gamma}_{X(\omega),\alpha}([t_k,t_{k+1}])}\leq p_1(\omega,[t_k,t_{k+1}])\norm{y_{t_{k}}}_\alpha+p_2(\omega,[t_k,t_{k+1}])\\
			&\leq p_1(\omega,[t_k,t_{k+1}])\norm{y,y^\prime}_{\cD^{2\gamma}_{X(\omega),\alpha}([t_{k-1},t_{k}])}+p_2(\omega,[t_k,t_{k+1}])\\
			&\leq p_1(\omega,[t_k,t_{k+1}])(p_1(\omega,[t_{k-1},t_{k}])\norm{y_{t_{k-1}}}_\alpha+p_2(\omega,[t_{k-1},t_{k}]))+p_2(\omega,[t_k,t_{k+1}])\\
			&\leq \prod_{l=0}^k p_1(\omega,[t_l,t_{l+1}])\norm{y_{s}}_\alpha+\sum_{n=0}^{k} p_2(\omega,[t_{k-n},t_{k-n+1}])\prod_{l=1}^n p_1(\omega,[t_{k-n+l},t_{k-n+l+1}])\\
			&\leq p_1(\omega,[s,t])^{\widetilde{N}}\norm{y_{s}}_\alpha+p_2(\omega,[s,t])\sum_{l=0}^{\widetilde{N}-1} p_1(\omega,[s,t])^l,
		\end{align*}
		since $p_1(\omega,[t_k,t_{k+1}])\leq p_1(\omega,[s,t])$ and analogously for $p_2$, for $k<\widetilde{N}-1$. Together with \eqref{est:DiscreteGubNorm}, this provides \eqref{ineq:VarRiedelEst}.
		\qed \\
	\end{proof}
	
	For a better readability, we define 
	\begin{align*}
		\begin{split}
			\widetilde{P}(\omega,[s,t])&:=M N_{s,t}(\omega)(1+\left[X(\omega)\right]_{\gamma,[s,t]})e^{N_{s,t}(\omega)\widetilde{M}},\quad
			P_1(\omega,[s,t]):=\widetilde{N}\widetilde{P}(\omega,[s,t])^{\widetilde{N}+1},\\
			P_2(\omega,[s,t])&:=M \widetilde{N}N_{s,t}(\omega)(1+\left[X(\omega)\right]_{\gamma,[s,t]})\frac{e^{N_{s,t}(\omega)\widetilde{M}+\widetilde{M}}-1}{e^{\widetilde{M}}-1}P(\left[X(\omega)\right]_{\gamma,[s,t]},\left[\xx(\omega)\right]_{2\gamma,\Delta_{[s,t]}})\\
			&\times\frac{\widetilde{P}(\omega,[s,t])^{\widetilde{N}}-1}{\widetilde{P}(\omega,[s,t])-1}.
		\end{split}
	\end{align*}
	With those abbreviations, the \eqref{ineq:VarRiedelEst} for the solution $(y,y^\prime)$ of \eqref{eq:MainEq} becomes
	\begin{align}\label{est:SolEst}
		\norm{y,y^\prime}_{\cD^{2\gamma}_{X(\omega),\alpha}([s,t])}\leq \norm{y_s}_\alpha P_1(\omega,[s,t])+P_2(\omega,[s,t]).
	\end{align}
	\begin{remark}\label{rem:ConstantsSolEst}
		\begin{itemize}
			\item[i)] We do not highlight the dependence of $\widetilde{N}$ on the length of the interval $[s,t]$, since we will later use these estimates only for intervals of length one. In this case, we have $\widetilde{N}=\lceil d^{-1}\rceil$, where we recall that $d=(4\widetilde{M})^{-\frac{1}{1-\max\{\sigma_F,2\gamma\}}}<1$. 
			\item[ii)] To our aims, the dependence of the constants in~\eqref{ineq:VarRiedelEst} on the nonlinear terms $F$ and $G$ is crucial since these influence the long-time behavior of~\eqref{eq:MainEq}. 
			One can observe that $\widetilde{M}$ has the form $\widetilde{M}=\max\{C_A,C_F,C_G\}\overline{M}$, where $\overline{M}$ does depend on $\eta,\gamma$ and $\sigma_G$, as well as on $F$ and $\lambda_A$. Modifying the sewing lemma \cite[Theorem 4.1]{GHN2021} in order to incorporate the control $W_{s,t}(\omega)$ instead of the H\"older norm of the corresponding rough path leads to this constant.
		\end{itemize}
	\end{remark}
	The main reason why we use the solution estimate \eqref{ineq:VarRiedelEst} is its integrability, in contrast to the bound derived in~\cite{HN22} and used in~\cite{KN23}. Moreover, it allows us to quantify the random radius of the absorbing set for~\eqref{eq:MainEq}, as specified in Definition~\ref{def:absorbing}, capturing its long-time behavior. To this aim the following condition on the noise is necessary, which was verified for the rough path lift of the fractional Brownian motion with Hurst parameter $H\in (\frac{1}{4},\frac{1}{2})$ in \cite[Proposition 2.9]{RiedelVarzaneh23}.
	\begin{assumptions}\label{ass:Noise}
		Let $X$ be a Gaussian process defined on an abstract Wiener space with associated Cameron-Martin space $\mathcal{H}$, such that it can be enhanced to a geometric $\gamma$-H\"older rough cocycle $\X:=(X,\XX)$. 
		For every $h\in \mathcal{H}$ we assume
		\begin{align*}
			\sup\limits_{\cP\subset [s,t]} \left\{\sum_{j=0}^{\abs{\cP}} \abs{h_{\kappa_{j+1}}-h_{\kappa_j}}^{\frac{1}{\gamma^\prime}} \right\}< \infty,
		\end{align*}
		where the supremum is taken over all partitions $\cP=\{s=\kappa_0<\ldots < \kappa_n <t\}$ of $[s,t]$ and
		\begin{align*}
			W_{\mathbf{h},\gamma^\prime,\eta}(0,1)\lesssim \norm{h}_{\mathcal{H}}^{\frac{1}{\gamma^\prime-\eta}},
		\end{align*}
		where $\gamma^\prime>0$ satisfies $\gamma+\gamma^\prime -2\eta>1$.
	\end{assumptions}

	\begin{lemma}{\em (\cite[Theorem 2.13 iii)]{RiedelVarzaneh23}).}\label{lem:IntegrableBound}
		Under the assumptions of Theorem \ref{thm:VarRiedelEst} and Assumption \ref{ass:Noise} the bound of $\norm{y,y^\prime}_{\mathcal{D}^{2\gamma}_{X,\alpha}([s,t])}$ in \eqref{ineq:VarRiedelEst} is integrable.~In particular, we have $N_{s,t}(\cdot)(1+[X(\cdot)]_{\gamma,[s,t]})e^{N_{s,t}(\cdot)\widetilde{M}}\in \bigcap_{p\geq 1} L^p(\Omega)$ which leads to
		\begin{align}\label{eq:IntegrableBound}
			P_1(\cdot, [s,t]),P_2(\cdot,[s,t]) \in \bigcap_{p\geq 1} L^p(\Omega).
		\end{align}
	\end{lemma}
	\begin{remark}\label{rem:RelaxAssump}
		For simplicity, the assumptions on the diffusion coefficient $G$ in \eqref{eq:MainEq} are slightly more restrictive compared to~\cite{GHN2021,HN22}. First, we require the derivatives of $G:E_{\alpha-\vartheta}\to E_{\alpha-\vartheta-\sigma_G}$ to be bounded for all values of $\vartheta\in [0,2\gamma]$. Additionally, $G$ must be a bounded function. Both of these limitations are the outcome of specific computations presented in \cite{RiedelVarzaneh23}, where they are used to modify the sewing lemma in order to incorporate the control $W$ instead of the H\"older norms of the rough path. As hinted in \cite[Remark 2.14]{RiedelVarzaneh23}, this condition can be replaced by the boundedness of the derivative of $DG(\cdot)\circ G(\cdot):E_{\alpha-\gamma}\to E_{\alpha-2\gamma-\sigma_G}$ as established in \cite{HN22}. 
	\end{remark}
	
	\subsection{Random dynamics}
	Since our goal is to investigate the long-time behavior of the solution to \eqref{eq:MainEq} we first introduce basic concepts from the theory of random dynamical systems \cite{Arnold}. 
	\begin{definition}
		A continuous random dynamical system on a separable Banach space $E$ over a metric dynamical system $(\Omega,\mathcal{F},\P,(\theta_t)_{t\in\R})$ is a mapping 
		\[\phi:\R^+\times \Omega\times E\to E, (t,\omega,x)\mapsto \phi(t,\omega,x),\]
		which is $(\mathcal{B}(\R^+)\otimes \mathcal{F}\otimes \mathcal{B}(E),\mathcal{B}(E))$-measurable and satisfies
		\begin{itemize}
			\itemsep -4pt
			\item[i)] $\phi(0,\omega,\cdot)=\Id_E$ for every $\omega\in \Omega$,
			\item[ii)] $\phi(t+s,\omega,x)=\phi(t,\theta_s \omega,\phi(s,\omega,x))$ for all $\omega\in \Omega$, $t,s\in \R^+$ and $x\in E$,					
			\item[iii)] the map $\phi(t,\omega,\cdot):E\to E$ is for every $t\in \R^+$ and $\omega\in \Omega$ continuous.
		\end{itemize}
	\end{definition}
	Condition ii) is referred to as the (perfect) cocycle property. Provided that the driving noise forms a rough path cocycle, we refer to~\cite{BRiedelScheutzow,Hofmanova2,HN20,HN22} for the generation of a random dynamical system from a rough (partial) differential equation.
	\begin{theorem}{\em (\cite[Theorem 3.12]{HN22}).}
		Under the Assumptions \ref{ass} and \ref{ass:Nonlin}, the solution operator of \eqref{eq:MainEq} with $y_0\in E_\alpha$ driven by a rough path cocycle generates a continuous random dynamical system on $E_\alpha$.
	\end{theorem}
	
	\begin{definition}\label{def:Tempered}
		\begin{itemize}
			\item[i)] A family of non-empty closed sets $B:=\{B(\omega)\}_{\omega\in \Omega}$ in $E$ is called a random set if 
			\begin{align*}
				\omega\mapsto \inf_{y\in B(\omega)} \norm{x-y}_E,
			\end{align*}
			is a random variable for all $x\in E$. It is called further a bounded, or compact, random set if the sets $B(\omega)$ are all bounded, respectively compact, for all $\omega\in \Omega$.
			\item[ii)] A random variable $Y:\Omega\to [0,\infty)$ is called tempered with respect to $(\theta_t)_{t\in \R}$ if 
			\begin{align*}
				\lim_{t\to \pm \infty} e^{-\beta \abs{t}} Y(\theta_t \omega)=0,
			\end{align*}
			for all $\beta>0$ and $\omega\in \Omega$. If $B:=\{B(\omega)\}_{\omega\in \Omega}$ is a bounded random set and $\omega \mapsto \sup\limits_{x\in B(\theta_{-t} \omega)}\norm{x}_E$ is tempered, then $B$ is called a tempered set.
		\end{itemize}
	\end{definition}
	Note that the temperedness defined in Definition \ref{def:Tempered} is equivalent to a subexponential growth condition, i.e.
	\begin{align}\label{eq:EquivTempered}
		\lim_{t\to \pm \infty} \frac{\log^+(Y(\theta_t \omega))}{\abs{t}}=0.
	\end{align}
	In the following we denote by $\mathscr{D}$ the universe of tempered sets in $E_\alpha$.
	\begin{definition}\label{def:attractor}
		A random set $\cA:=\{\cA(\omega)\}_{\omega\in \Omega}\in \mathscr{D}$ is called a random pullback $\mathscr{D}$-attractor for the random dynamical system $\phi$ if
		\begin{itemize}
			\item[i)] $\cA(\omega)$ is compact for every $\omega\in \Omega$,
			\item[ii)] $\cA$ is $\phi$-invariant, that means for every $t\geq 0$ and $\omega \in \Omega$ we have
			\begin{align*}
				\phi(t,\omega,\cA(\omega))=\cA(\theta_t\omega),
			\end{align*}
			\item[iii)] $\cA$ pullback attracts every tempered random set $B=\{B(\omega)\}_{\omega\in \Omega}$, that means 
			\begin{align*}
				\lim\limits_{t\to \infty} d(\phi(t,\theta_{-t}\omega,B(\theta_{-t}\omega)),\cA(\omega))=0,
			\end{align*}
			where $d(A_1,A_2):=\sup\limits_{x\in A_1} \inf\limits_{y\in A_2} \norm{x-y}_E$ for $A_1,A_2\subset E$ is the Hausdorff semimetric. 
		\end{itemize}
	\end{definition}
	Since it may be difficult to verify the attracting property, there exists a sufficient condition which ensures this, provided that there exists an absorbing set for the random dynamical system as specified below.  
	\begin{definition}\label{def:absorbing}
		A random set $B=\{B(\omega)\}_{\omega\in \Omega}\in \mathscr{D}$ is called random pullback $\mathscr{D}$-absorbing if for every $D=\{D(\omega)\}_{\omega\in \Omega}\in \mathscr{D}$ and $\omega\in \Omega$ there exists an absorbing time $T_D(\omega)>0$ such that
		\begin{align*}
			\phi(t,\theta_{-t} \omega, D(\theta_{-t}\omega))\subset B(\omega),
		\end{align*}
		for all $t\geq T_D(\omega)$.
	\end{definition}
	A suitable way to prove the existence of such an absorbing set, is to show the existence of a positive tempered random variable $R$, such that for any $y_0(\theta_{-t}\omega)\in D(\theta_{-t} \omega)$ with $D(\omega)\in \mathscr{D}$ and $\omega\in \Omega$ the estimate
	\begin{align*}
		\limsup\limits_{t\to \infty} \norm{\phi(t,\theta_{-t} \omega,y_0(\theta_{-t}\omega))}_\alpha\leq R(\omega)
	\end{align*}
	holds. Then the open ball $B(\omega):=B(0,R(\omega)+\overline{\delta})$, for some constant $\overline{\delta}>0$, is a random pullback $\mathscr{D}$-absorbing set.
	\begin{theorem}{\em (\cite[Theorem 3.5]{FlandoliSchmalfuss96}).}\label{thm:ExAttractor}
		Assume the existence of a compact set $B:=\{B(\omega)\}_{\omega\in \Omega}\in \mathscr{D}$ which is random pullback $\mathscr{D}$-absorbing. Then the continuous random dynamical system $\phi$ has an unique random pullback $\mathscr{D}$-attractor $\cA:=\{\cA(\omega)\}_{\omega\in \Omega}$ given by
		\begin{align*}
			\cA(\omega):=\bigcap_{s\geq 0}\overline{\bigcup_{t\geq s} \phi(t,\theta_{-t} \omega,B(\theta_{-t}\omega))}.
		\end{align*}
	\end{theorem}
	\subsection{Gronwall lemmata and the Mittag-Leffler function}
	Due to Assumption \ref{ass:Nonlin} 1) we need the following version of Gronwall's inequality in order to derive an a-priori estimate of the solution of~\eqref{eq:MainEq} in $E_\alpha$.
	\begin{lemma}{\em (Singular Henry-Gronwall inequality~\cite[Lemma 7.1.1]{Henry81}).}
		Let $v,h\in L^\infty_{{\rm loc}}([0,T);[0,\infty))$ be non-negative functions satisfying
		\begin{align*}
			v(t)\leq h(t)+M\int_0^t (t-r)^{\beta-1}v(r)~\txtd r,
		\end{align*}
		for $t\in (0,T),T\in (0,\infty]$ and $M,\beta>0$. Then one has
		\begin{align}\label{ineq:singularGronwall2}
			v(t)\leq h(t)+(\Gamma(\beta)M)^{\beta^{-1}}\int_0^t h(r) E^\prime_{\beta,1}((t-r)(\Gamma(\beta)M)^{\beta^{-1}})~\txtd r,
		\end{align}
		where, $\Gamma(z):=\int_0^\infty e^{-r}r^{z-1} ~\txtd r$ is the Gamma function and  $E_{\beta,c}(z):=\sum_{k=0}^\infty \frac{z^{\beta k}}{\Gamma(k\beta+c)}$  the Mittag-Leffler function. 
	\end{lemma}
	This fact is a generalization of the classical Gronwall lemma, which is a special case of the previous statement for $\beta=1$ since $E_{1,1}(z)=e^z$. In order to prove now the existence of a random absorbing set later on, we need to investigate the long-time behavior of the solution. This incorporates also the derivative of the Mittag-Leffler function we obtain here due to the singular Gronwall-Henry inequality. But as a generalization of the exponential function, its asymptotic behavior is similar. This is crucial for the computation of the absorbing set in the next section.  
	\begin{lemma}
		For $\beta\in (0,1)$, there exist a constant $M_\beta>0$, such that for $z>1$ large enough we have
		\begin{align}\label{eq:MittagLefflerBound}
			E^\prime_{\beta,1}(z)\leq M_\beta e^{2z}.
		\end{align}
	\end{lemma}
	\begin{proof}
		Due to \cite[(94.19)]{SellYou} there exist two polynomials $P$ and $Q$ such that 
		\begin{align*}
			E_{\beta,\beta}(z)\leq P(z)+Q(z)e^z,
		\end{align*}
		for $z>1$. Since we can bound every polynomial by $e^z$, if $z$ is large enough, this leads to
		\begin{align*}
			E_{\beta,\beta}(z)\leq M_\beta e^{2 z},
		\end{align*}
		where $M_\beta$ depends on the coefficients of $P$ and $Q$. 
		This further means that
		\[ \lim\limits_{t\to\infty}\frac{1}{t} \log E_{\beta,\beta}(\mu t)=\mu,  \]
		for all $\mu,\beta>0$. This is crucial for our aims since we are interested in the long-time behavior of~\eqref{eq:MainEq}.
		A similar statement holds also for the derivative of the Mittag-Leffler's function.  
		In order to obtain this, note that 
		\begin{align*}
			E^\prime_{\beta,1}(z)&=\sum_{k=1}^\infty \frac{k \beta}{\Gamma(k\beta+1)}z^{k\beta-1}=\sum_{k=1}^\infty \frac{z^{k\beta-1}}{\Gamma(k\beta)}
			=\frac{1}{z} \sum_{k=1}^\infty \frac{z^{(k-1)\beta+\beta}}{\Gamma((k-1)\beta+\beta)}\\
			&=\frac{z^\beta}{z}\sum_{k=0}^\infty \frac{z^{k\beta}}{\Gamma(k\beta+\beta)}=z^{\beta-1}E_{\beta,\beta}(z).
		\end{align*}
		Therefore, the asymptotic behavior of $E^\prime_{\beta,1}$ is the same as the one of $E_{\beta,\beta}$, where the constant $M_\beta$ can be varied.
		\qed\\
	\end{proof}
	\begin{lemma}{\em (discrete Gronwall inequality~\cite[Lemma 3.12]{DucHong23}).}\label{lem:discreteGronwall}
		Let $(u_n)_{n\in \N}$, $(b_{k})_{k\in \N_0}$ and $(c_{k})_{\in \N_0}$  be a non-negative sequences and $a\geq 0$, satisfying
		\begin{align*}
			u_n\leq a+\sum_{k=0}^{n-1} b_{k} u_k +\sum_{k=0}^{n-1} c_{k},
		\end{align*}
		for all $n\in \N$. Then we have 
		\begin{align}\label{ineq:discreteGronwall}
			u_n\leq \max\{a,u_0\}\prod_{j=0}^{n-1} (1+b_{j})+\sum_{k=0}^{n-1}c_{k} \prod_{j=k+1}^{n-1} (1+b_{j}),
		\end{align}
		for all $n\in \N$. 
	\end{lemma}

	\section{Random attractor}\label{sec:main}
	Before we show the existence of a random attractor, let us first collect some integral estimates needed later. Recall, that $\X(\omega)=(X(\omega),\XX(\omega))$ is a geometric rough path cocycle satisfying Assumption \ref{ass:Noise}. From now on, let $(y,y^\prime)\in \cD^{2\gamma}_{X(\omega),\alpha}$ be the unique global solution of \eqref{eq:MainEq} driven by $\X(\omega)$.
	\begin{lemma}\label{lem:IntegralEstimates} Let $t\geq 0$. Then we have 
		\begin{align*}
			\norm{\int_0^t S_{t-r}G(y_r)~\txtd \X_r(\omega)}_\alpha&\leq C_GC_1 \sum_{l=0}^{\lfloor t \rfloor}e^{-\lambda_A(t-l-1)}P_3(\omega,[l,l+1]),\\
			\norm{\int_0^t S_{t-r}F(y_r)~\txtd r}_\alpha&\leq C_{-\sigma_F} C_F\int_0^t e^{-\lambda_A(t-r)} (t-r)^{-\sigma_F}\norm{y_r}_{\alpha}~\txtd r+C_2,
		\end{align*}
		where we introduced the expressions
		\begin{align*}
			C_1&:=\max\{C_IC_A,C_I\}, \\ 
			C_2&:=C_{-\sigma_F}C_F \lambda_A^{\sigma_F-1} \Gamma(1-\sigma_F),\\
			P_3(\omega,[l,l+1])&:=\rho_{\gamma,[l,l+1]}(\X(\omega))^2(1+\norm{y,y^\prime}_{\cD^{2\gamma}_{X,\alpha}([l,l+1])}).
		\end{align*}
	\end{lemma}
	\begin{proof}
		Similar to \cite{DucHong23,KN23} we split the rough integral into integrals over an interval with length one. Regarding Assumption \ref{ass} 2), \eqref{estimate:integral} and \eqref{ineq:Nonlin} we get 
		\begin{align*}
			\norm{\int_0^t S_{t-r}G(y_r)~\txtd \X_r(\omega)}_\alpha &\leq  \sum_{l=0}^{\lfloor t \rfloor-1}\norm{\int_l^{l+1} S_{t-r}G(y_r)~\txtd \X_r(\omega)}_\alpha+\norm{\int_{\lfloor t \rfloor}^t S_{t-r}G(y_r)~\txtd \X_r(\omega)}_\alpha\\
			&\leq C_A \sum_{l=0}^{\lfloor t \rfloor-1}e^{-\lambda_A(t-l-1)}\norm{\int_l^{l+1} S_{l+1-r}G(y_r)~\txtd \X_r(\omega)}_\alpha\\
			&+\norm{\int_{\lfloor t \rfloor}^t S_{t-r}G(y_r)~\txtd \X_r(\omega)}_\alpha\\
			&\leq C_I C_A \sum_{l=0}^{\lfloor t \rfloor-1}e^{-\lambda_A(t-l-1)}\rho_{\gamma,[l,l+1]}(\X(\omega))\norm{G(y),(G(y))^\prime}_{\cD^{2\gamma}_{X(\omega),\alpha-\sigma_G}([l,l+1])}\\
			&+C_I\rho_{\gamma,[\lfloor t \rfloor,t]}(\X(\omega))\norm{G(y),(G(y))^\prime}_{\cD^{2\gamma}_{X(\omega),\alpha-\sigma_G}([\lfloor t \rfloor,t])}(t-\lfloor t \rfloor)^{\gamma-\sigma_G}\\
			&\leq C_GC_1 \sum_{l=0}^{\lfloor t \rfloor}e^{-\lambda_A(t-l-1)}\rho_{\gamma,[l,l+1]}(\X(\omega))^2(1+\norm{y,G(y)}_{\cD^{2\gamma}_{X(\omega),\alpha}([l,l+1])}),
		\end{align*}
		which shows the first inequality. The second integral can be estimated regarding the Lipschitz continuity of $F:E_\alpha\to E_{\alpha-\sigma_F}$
		\begin{align*}
			\norm{\int_0^t S_{t-r}F(y_r)~\txtd r}_\alpha&\leq C_{-\sigma_F}C_F\int_0^t e^{-\lambda_A(t-r)} (t-r)^{-\sigma_F}(1+\norm{y_r}_{\alpha})~\txtd r \\
			& \leq C_{-\sigma_F} C_F\int_0^t e^{-\lambda_A(t-r)} (t-r)^{-\sigma_F}\norm{y_r}_{\alpha}~\txtd r\\ &+C_{-\sigma_F} C_F\int_0^t e^{-\lambda_A(t-r)} (t-r)^{-\sigma_F}~\txtd r,
		\end{align*}
		where the last term can be bounded by $\lambda_A^{\sigma_F-1} \Gamma(1-\sigma_F)$, which can be seen by a simple substitution and the definition of the Gamma function.
		\qed\\
	\end{proof}
	\begin{lemma}\label{lem:SolEst}
		Let $n\in \N_0$ and define 
		\begin{align*}
			\widetilde{C}_1=C_1e^{\lambda_A}\max\left\{\widetilde{L}, \frac{M_{1-\sigma_F}}{2}\right\},\quad
			\widetilde{C}_2=C_2\left( \widetilde{L}+\frac{LM_{1-\sigma_F}}{2\lambda}\right), \quad
			\widetilde{C}_A=C_A\max\left\{\widetilde{L}, \frac{M_{1-\sigma_F}}{2}\right\}, 
		\end{align*}
		where $L:=2(C_{-\sigma_F}C_F\Gamma(1-\sigma_F))^{\frac{1}{1-\sigma_F}}$, $\widetilde{L}:=\frac{2E_{1-\sigma_F,1}\left(t_0\frac{L}{2}\right)}{L}$+1, $\lambda:=\lambda_A-L$ and $t_0>0$ such that $t_0\frac{L}{2}$ is large enough for \eqref{eq:MittagLefflerBound} to hold. Then we get for $t\in [n,n+1]$ the estimate 
		\begin{align}\label{ineq:AprioriEst}
			\norm{y_t}_\alpha e^{\lambda t}\leq \widetilde{C}_A \norm{y_0}_\alpha+\widetilde{C}_2e^{\lambda t}+\widetilde{C}_1C_G \sum_{l=0}^{n}e^{\lambda l}P_3(\omega,[l,l+1]).
		\end{align}
	\end{lemma}
	\begin{proof}
		First assume $t\in [n,n+1)$. Due to \eqref{eq:mildSolution} we can estimate the path component $y_t$ of the solution by
		\begin{align}\label{eq:ProofSolEstimate1}
			\norm{y_t}_\alpha\leq \norm{S_t y_0}_\alpha +\norm{\int_0^t S_{t-r}F(y_r)~\txtd r}_\alpha +\norm{\int_0^t S_{t-r}G(y_r)~\txtd \X_r(\omega)}_\alpha.
		\end{align}
		We now multiply \eqref{eq:ProofSolEstimate1} by $e^{\lambda_A t}$ and get together with \eqref{hg:1} and Lemma \ref{lem:IntegralEstimates}
		\begin{align*}
			\norm{y_t}_\alpha e^{\lambda_A t}&\leq C_A \norm{y_0}_\alpha +C_{2} e^{\lambda_A t}+C_GC_1 \sum_{l=0}^{n}e^{\lambda_A(l+1)}P_3(\omega,[l,l+1])\\
			&+C_{-\sigma_F} C_F\int_0^t  (t-r)^{-\sigma_F}e^{\lambda_A r}\norm{y_r}_{\alpha}~\txtd r,
		\end{align*}
		which fits in the setting of the singular Gronwall-Henry lemma. In conclusion, we now apply Lemma \ref{ineq:singularGronwall2} to $v(r):=\norm{y_r}_\alpha e^{\lambda_A r}$ and obtain 
		\begin{align}\label{eq:ProofSolEstimate2}
			\begin{split}
				&\norm{y_t}_\alpha e^{\lambda_A t}\leq C_A \norm{y_0}_\alpha +C_{2} e^{\lambda_A t}+C_GC_1 \sum_{l=0}^{n}e^{\lambda_A(l+1)}P_3(\omega,[l,l+1])\\
				&+\frac{L}{2}\int_{0}^t \left(C_A \norm{y_0}_\alpha +C_{2} e^{\lambda_A r}+C_GC_1 \sum_{l=0}^{\lfloor r \rfloor}e^{\lambda_A(l+1)}P_3(\omega,[l,l+1])\right)E^\prime_{1-\sigma_F,1}\left((t-r)\frac{L}{2}\right)~\txtd r,
			\end{split}
		\end{align}
		with $L:=2(C_{-\sigma_F}C_F\Gamma(1-\sigma_F))^{\frac{1}{1-\sigma_F}}$. Further, due to the time-continuity of $y$, this also holds for $t=n+1$. Now we have to bound the derivative of the Mittag-Leffler function. Since the estimate~\eqref{eq:MittagLefflerBound} is only valid for large values, we consider $t_0>0$ such that $t_0\frac{L}{2}$ is large enough for \eqref{eq:MittagLefflerBound} to hold and introduce 
		$$h(r):=C_A \norm{y_0}_\alpha +C_{2} e^{\lambda_A r}+C_GC_1 \sum_{l=0}^{\lfloor r \rfloor}e^{\lambda_A(l+1)}P_3(\omega,[l,l+1]).$$
		For $t>t_0$ we split the integral as follows and use that $h$ is monotonously increasing
		\begin{align*}
			&\int_0^t h(r)E^\prime_{1-\sigma_F,1}\left((t-r)\frac{L}{2}\right)~\txtd r=\int_{0}^{t-t_0} h(r)E^\prime_{1-\sigma_F,1}\left((t-r)\frac{L}{2}\right)~\txtd r \\
			&+ \int_{t-t_0}^t h(r)E^\prime_{1-\sigma_F,1}\left((t-r)\frac{L}{2}\right)~\txtd r\\
			&\leq M_{1-\sigma_F}\int_{0}^{t-t_0} h(r) e^{(t-r)L}~\txtd r+\int_{0}^{t_0} h(t-r)E^\prime_{1-\sigma_F,1}\left(r\frac{L}{2}\right)~\txtd r\\
			&\leq M_{1-\sigma_F}\int_{0}^{t} h(r) e^{(t-r)L}~\txtd r+h(t)\int_{0}^{t_0} E^\prime_{1-\sigma_F,1}\left(r\frac{L}{2}\right)~\txtd r\\
			&=M_{1-\sigma_F}\int_{0}^{t} h(r) e^{(t-r)L}~\txtd r+h(t) \frac{2E_{1-\sigma_F,1}\left(t_0\frac{L}{2}\right)}{L}.
		\end{align*}
		We set $\widetilde{L}:=\frac{2E_{1-\sigma_F,1}\left(t_0\frac{L}{2}\right)}{L}+1$ and multiply \eqref{eq:ProofSolEstimate2} by $e^{-Lt}$ to get
		\begin{align*}
			\norm{y_t}_\alpha e^{\lambda t}&\leq \widetilde{L}\left(C_A \norm{y_0}_\alpha e^{-Lt}+C_{2} e^{\lambda t}+C_GC_1 \sum_{l=0}^{n}e^{-Lt}e^{\lambda_A(l+1)}P_3(\omega,[l,l+1])\right)\\
			&+\frac{LM_{1-\sigma_F}}{2}\int_{0}^t \left(C_A \norm{y_0}_\alpha +C_{2} e^{\lambda_A r}+C_GC_1 \sum_{l=0}^{\lfloor r \rfloor}e^{\lambda_A(l+1)}P_3(\omega,[l,l+1])\right)e^{-Lr}~\txtd r,
		\end{align*}
		which makes it easier to evaluate the integrals in the second line, where the only non-trivial is the third one.  By Fubini's theorem we obtain
		\begin{align*}
			\int_{0}^t \sum_{l=0}^{\lfloor r \rfloor}e^{\lambda_A(l+1)}P_3(\omega,[l,l+1])e^{-Lr}~\txtd r&\leq \sum_{l=0}^{n} e^{\lambda l}e^{\lambda_A}P_3(\omega,[l,l+1])\int_{l}^t e^{-L (r-l)}~\txtd r\\
			&= \sum_{l=0}^{n} e^{\lambda l}e^{\lambda_A}P_3(\omega,[l,l+1])\left(\frac{1-e^{-L(t-l)}}{L} \right).
		\end{align*}
		Putting all these estimates together, this leads to
		\begin{align*}
			\norm{y_t}_\alpha e^{\lambda t}&\leq C_A \norm{y_0}_\alpha\left(\widetilde{L}e^{-Lt} +\frac{M_{1-\sigma_F}}{2}(1-e^{-Lt})\right)+C_2\left(\widetilde{L}e^{\lambda t} +\frac{LM_{1-\sigma_F}}{2\lambda}(e^{\lambda t}-1)\right)\\
			&+C_GC_1 \sum_{l=0}^{n} e^{\lambda l}e^{\lambda_A}P_3(\omega,[l,l+1])\left(\widetilde{L}e^{-L(t-l)}+\frac{M_{1-\sigma_F}}{2}(1-e^{-L(t-l)}) \right)\\
			&\leq \widetilde{C}_A \norm{y_0}_\alpha+C_2e^{\lambda t}\left(\widetilde{L} +\frac{LM_{1-\sigma_F}}{2\lambda}(1-e^{-\lambda t})\right)+\widetilde{C}_1C_G \sum_{l=0}^{n}e^{\lambda l}P_3(\omega,[l,l+1])\\
			&\leq \widetilde{C}_A \norm{y_0}_\alpha+\widetilde{C}_2e^{\lambda t}+\widetilde{C}_1C_G \sum_{l=0}^{n}e^{\lambda l}P_3(\omega,[l,l+1]),
		\end{align*}
		with $\lambda=\lambda_A-L$ and $\widetilde{C}_1=C_1e^{\lambda_A}\max\left\{\widetilde{L}, \frac{M_{1-\sigma_F}}{2}\right\},
		\widetilde{C}_2=C_2\left( \widetilde{L}+\frac{LM_{1-\sigma_F}}{2\lambda}\right),
		\widetilde{C}_A=C_A\max\left\{\widetilde{L}, \frac{M_{1-\sigma_F}}{2}\right\}$.
		\qed\\
	\end{proof}
	
	The right-hand side of \eqref{ineq:AprioriEst} still depends on the solution $y$, via $P_3(\omega,[l,l+1])$. However, combining the discrete Gronwall lemma with \eqref{est:SolEst}, it is possible to get an estimate where the right-hand side is independent of the solution.
	First, recall that 
	\begin{align*}
		P_3(\omega,[l,l+1])&=\rho_{\gamma,[l,l+1]}(\X(\omega))^2(1+\norm{y,y^\prime}_{\cD^{2\gamma}_{X,\alpha}([l,l+1])}),
	\end{align*}
	and that the controlled rough path norm of the solution $(y,y')=(y,G(y))$ can be estimated due to Theorem \ref{thm:VarRiedelEst} by
	\begin{align*}
		\norm{y,y^\prime}_{\cD^{2\gamma}_{X(\omega),\alpha}([l,l+1])}\leq \norm{y_l}_\alpha P_1(\omega,[l,l+1])+P_2(\omega,[l,l+1])
	\end{align*}
	for $l\in \N_0$.
	\begin{lemma}\label{lem:discreteEst}
		Let $n\in \N_0$. Then we get the estimate
		\begin{align*}
			\norm{y_n}_\alpha&\leq \widetilde{C}_A \norm{y_0}_\alpha e^{-\lambda n}\prod_{j=0}^{n-1} (1+H_1(\omega,[j,j+1])) +\sum_{k=0}^{n-1}e^{-\lambda (n-k)}H_2(\omega,[k,k+1])\\
			&\times \prod_{j=k+1}^{n-1} (1+H_1(\omega,[j,j+1])),
		\end{align*}
		where we define
		\begin{align}\label{eq:DefH}
			\begin{split}
				H_1(\omega,[l,l+1])&:=\widetilde{C}_1C_G \rho_{\gamma,[l,l+1]}(\X(\omega))^2 P_1(\omega,[l,l+1]),\\
				H_2(\omega,[l,l+1])&:=\max\{\widetilde{C}_Ae^\lambda,\widetilde{C}_1C_G\}(1+\rho_{\gamma,[l,l+1]}(\X(\omega))^2(1+P_2(\omega,[l,l+1]))).
			\end{split}
		\end{align}
	\end{lemma}
	\begin{proof}
		We apply Lemma \ref{lem:SolEst} for $t=n$ and obtain
		\begin{align*}
			\norm{y_n}_\alpha e^{\lambda n}&\leq \widetilde{C}_A \norm{y_0}_\alpha+\widetilde{C}_2e^{\lambda n}+\widetilde{C}_1C_G \sum_{l=0}^{n-1}e^{\lambda l}P_3(\omega,[l,l+1])\\
			&= \widetilde{C}_A \norm{y_0}_\alpha+\widetilde{C}_2e^{\lambda n}+\widetilde{C}_1C_G \sum_{l=0}^{n-1}e^{\lambda l}\rho_{\gamma,[l,l+1]}(\X(\omega))^2(1+\norm{y,y^\prime}_{\cD^{2\gamma}_{X(\omega),\alpha}})\\
			&= \widetilde{C}_A\norm{y_0}_\alpha+\widetilde{C}_2e^{\lambda n}+\widetilde{C}_1C_G \sum_{l=0}^{n-1} e^{\lambda l}\rho_{\gamma,[l,l+1]}(\X(\omega))^2(1+P_2(\omega,[l,l+1]))\\
			&+\widetilde{C}_1C_G \sum_{l=0}^{n-1} e^{\lambda l}\rho_{\gamma,[l,l+1]}(\X(\omega))^2\norm{y_l}_\alpha P_1(\omega,[l,l+1])\\
			&\leq \widetilde{C}_A\norm{y_0}_\alpha+\max\{\widetilde{C}_Ae^\lambda,\widetilde{C}_1C_G\} \sum_{l=0}^{n-1} e^{\lambda l}(1+\rho_{\gamma,[l,l+1]}(\X(\omega))^2(1+P_2(\omega,[l,l+1])))\\
			&+\widetilde{C}_1C_G \sum_{l=0}^{n-1} e^{\lambda l}\rho_{\gamma,[l,l+1]}(\X(\omega))^2\norm{y_l}_\alpha P_1(\omega,[l,l+1]).
		\end{align*}
		With this we can now use the discrete Gronwall Lemma \ref{lem:discreteGronwall} for $u_n:=\norm{y_n}_\alpha e^{\lambda n}$ to obtain
		\begin{align*}
			\norm{y_n}_\alpha e^{\lambda n}&\leq \widetilde{C}_A \norm{y_0}_\alpha \prod_{j=0}^{n-1} \left(1 +\widetilde{C}_1C_G \rho_{\gamma,[j,j+1]}(\X(\omega))^2 P_1(\omega,[j,j+1])\right)\\
			&+\sum_{k=0}^{n-1}\max\{\widetilde{C}_Ae^\lambda,\widetilde{C}_1C_G\}e^{\lambda k}(1+\rho_{\gamma,[k,k+1]}(\X(\omega))^2(1+P_2(\omega,[k,k+1])))\\
			&\times\prod_{j=k+1}^{n-1} \left(1 +\widetilde{C}_1C_G \rho_{\gamma,[j,j+1]}(\X(\omega))^2 P_1(\omega,[j,j+1])\right).
		\end{align*}
		\qed\\
	\end{proof}
	
	The last ingredient required for the existence of an absorbing set is based on ergodic properties of the noise. Since $\mathbf{X}$ is Gaussian all moments of $X$ and $\XX$ exists and in particular all moments of the respective H\"older seminorms. Regarding the ergodicity of the metric dynamical system $(\theta_t)_{t\in \R}$, Birkhoff's ergodic theorem leads to
	\begin{align}\label{ineq:ErgodicBound}
		\begin{split}
			\limsup\limits_{n\to \infty}\frac{1}{n}\sum_{j=1}^{n} \left[X(\theta_{-j}\omega)\right]_{\gamma,J}^q&= \E\left[\left[X\right]_{\gamma,J}^q\right]=:K_q\\
			\limsup\limits_{n\to \infty}\frac{1}{n}\sum_{j=1}^{n} \left[\XX(\theta_{-j}\omega)\right]_{\gamma,\Delta_J}^q&= \E\left[\left[\XX\right]_{2\gamma,\Delta_J}^q\right]=:\mathbb{K}_q,
		\end{split}
	\end{align}
	for every compact interval $J$ and $q\geq 1$. We further set $\mathbf{K}_q:=K_q+\mathbb{K}_q$ and prove our main result.  For a better comprehension we recall that $\widetilde{N}=\left\lceil(4\widetilde{M})^{\frac{1}{1-\max\{\sigma_F,2\gamma\}}}\right\rceil$, where $\widetilde{M}$ depends on $F,G$ as discussed in Remark \ref{rem:ConstantsSolEst}, $\widetilde{C}_1=C_1e^{\lambda_A}\max\left\{\widetilde{L}, \frac{M_{1-\sigma_F}}{2}\right\}$ and define $$C(\widetilde{N}):=\max\{1+\widetilde{N},2(1+\widetilde{N}),2^{4(1+\widetilde{N})}M^{1+\widetilde{N}}3\}.$$
	Note that $\widetilde{N}$  determines the highest order  moment of $(X,\xx)$ that we must control. Since the noise was assumed to be Gaussian, the value of $\widetilde{N}$ is not important. 
	\begin{lemma}\label{lem:AbsorbingSet}
		Let Assumptions \ref{ass}, \ref{ass:Nonlin} and \ref{ass:Noise} be satisfied and further assume that 
		\begin{align}\label{ineq:AssumptionOnLambda}
			\lambda_A-2(C_{-\sigma_F}C_F\Gamma(1-\sigma_F))^{\frac{1}{1-\sigma_F}}>c (\mathbf{K}_q+1),
		\end{align}
		where $q$ and $c$ are given by
		\begin{align}\label{eq:ConstantQandC}
			c &:=C(\widetilde{N})\max\{\widetilde{M},\widetilde{C}_1C_G\}, \quad q:=\frac{4(1+\widetilde{N})}{\gamma-\eta}.
		\end{align}
		Then the random dynamical system $\phi$ associated to \eqref{eq:MainEq}, possesses a random pullback $\mathscr{D}$-absorbing set $B(\omega)$.
	\end{lemma}
	\begin{proof}
		We fix $D(\omega)\in \mathscr{D}$ and estimate $\phi(t,\theta_{-t}\omega,y_0(\theta_{-t}\omega))$ for $y_0(\omega)\in D(\omega)$. To this aim we combine Lemma \ref{lem:discreteEst} and \eqref{est:SolEst} to obtain for $t\in [n,n+1]$
		\begin{align*}
			\norm{\phi(t,\theta_{-t}\omega,y_0(\theta_{-t}\omega))}_\alpha &\leq \norm{y_n(\theta_{-t}\omega)}_\alpha P_1(\theta_{-t}\omega,[n,n+1])+P_2(\theta_{-t}\omega,[n,n+1])\\
			&\leq \widetilde{C}_A \norm{y_0(\theta_{-t}\omega)}_\alpha P_1(\theta_{-t}\omega,[n,n+1]) e^{-\lambda n}\prod_{j=0}^{n-1} (1+H_1(\theta_{-t}\omega,[j,j+1])) \\
			&+P_1(\theta_{-t}\omega,[n,n+1])\sum_{k=0}^{n-1}e^{-\lambda (n-k)}H_2(\theta_{-t}\omega,[k,k+1])\\
			&\times \prod_{j=k+1}^{n-1} (1+H_1(\theta_{-t}\omega,[j,j+1]))+P_2(\theta_{-t}\omega,[n,n+1]).
		\end{align*}
		Due to \ref{lem:GreedTime} iii) we have $P_i(\theta_\tau \omega,[s,t])=P_i(\omega,[s+\tau,t+\tau])$ as well as $H_i(\theta_\tau \omega,[s,t])=H_i(\omega,[s+\tau,t+\tau])$ for $i=1,2$ and $\tau\in \R$. Applying this to the previous estimate, we obtain
		\begin{align*}
			\norm{\phi(t,\theta_{-t}\omega,y_0(\theta_{-t}\omega))}_\alpha &\leq \widetilde{C}_A \norm{y_0(\theta_{-t}\omega)}_\alpha P_1(\omega,[-1,1]) e^{-\lambda n}\sup\limits_{\varepsilon\in [0,1]}\prod_{j=1}^{n} (1+H_1(\theta_{-j}\omega,[-\varepsilon,1-\varepsilon])) \\
			&+P_1(\omega,[-1,1])\sup\limits_{\varepsilon\in [0,1]} \sum_{k=1}^{\infty}e^{-\lambda k} H_2(\theta_{-k}\omega,[-\varepsilon,1-\varepsilon])\\
			&\times \prod_{j=1}^{k-1} (1+H_1(\theta_{-j} \omega,[\varepsilon,1-\varepsilon]))+P_2(\omega,[-1,1]).
		\end{align*}
		For the first term one can use the fact that $y_0$ is tempered, and therefore $e^{-\kappa t} \norm{y_0(\theta_t \omega)}_\alpha\to 0$ for $t\to \infty$ and $\kappa>0$. To compute this $\kappa$, we note that $\log(1+ae^b)\leq a+b$, which leads to
		\begin{align*}
			\log(1+H_1(\omega,[s,t]))&\leq N_{s,t}(\omega)\widetilde{M}(1+\widetilde{N}) \\ &+ \widetilde{C}_1C_G \widetilde{N}\rho_{\gamma,[s,t]}(\X(\omega))^2(MN_{s,t}(\omega)(1+\left[X(\omega)\right]_{\gamma,[s,t]}))^{1+\widetilde{N}}.
		\end{align*}
		Using Lemma \ref{lem:GreedTime} and $(a+b)^p\leq 2^{p-1} (a^p+b^p)$ we can bound the noise terms to obtain 
		\begin{align*}
			&\rho_{\gamma,[s,t]}(\X(\omega))^2N_{s,t}(\omega)^{1+\widetilde{N}}(1+\left[X(\omega)\right]_{\gamma,[s,t]}))^{1+\widetilde{N}}\\
			&\leq 2^{4(1+\widetilde{N})}\left(\left[X(\omega)\right]_{\gamma,[s,t]}^{2}+\left[\xx(\omega)\right]^{2}_{2\gamma,\Delta_{[s,t]}}+\left[X(\omega)\right]_{\gamma,[s,t]}^{\frac{4(1+\widetilde{N})}{\gamma-\eta}}+\left[\xx(\omega)\right]^{\frac{2(1+\widetilde{N})}{\gamma-\eta}}_{2\gamma,\Delta_{[s,t]}}\right)
		\end{align*}
		Therefore we obtain
		\begin{align*}
			&\frac{1}{n}\log\left(\sup_{\varepsilon\in [0,1]}\prod_{j=1}^n (1+H_1(\theta_{-j} \omega,[-\varepsilon,1-\varepsilon]))\right)\leq \sup_{\varepsilon\in [0,1]}\frac{1}{n}\log\left(\prod_{j=1}^n (1+H_1(\theta_{-j} \omega,[-\varepsilon,1-\varepsilon]))\right)\\
			&\leq \widetilde{M}(1+\widetilde{N}) \sup_{\varepsilon\in [0,1]}\frac{1}{n}\sum_{j=1}^n N_{-\varepsilon,1-\varepsilon}(\theta_{-j}\omega)\\
			&+ 2^{4(1+\widetilde{N})}M^{1+\widetilde{N}}\widetilde{C}_1C_G\\
			&\times\sup_{\varepsilon\in [0,1]}\frac{1}{n}\sum_{j=1}^n \left(\left[X(\omega)\right]_{\gamma,[-\varepsilon,1-\varepsilon]}^{2}+\left[\xx(\omega)\right]^{2}_{2\gamma,\Delta_{[-\varepsilon,1-\varepsilon]}}+\left[X(\omega)\right]_{\gamma,[-\varepsilon,1-\varepsilon]}^{\frac{4(1+\widetilde{N})}{\gamma-\eta}}+\left[\xx(\omega)\right]^{\frac{2(1+\widetilde{N})}{\gamma-\eta}}_{2\gamma,\Delta_{[-\varepsilon,1-\varepsilon]}}\right) \\
			&\leq  \widetilde{M}(1+\widetilde{N})+\widetilde{M}(1+\widetilde{N})\chi^{-\frac{1}{\gamma-\eta}}\sup_{\varepsilon\in [0,1]}\frac{1}{n}\sum_{j=1}^n \left(\left[X(\omega)\right]_{\gamma,[-\varepsilon,1-\varepsilon]}^{\frac{1}{\gamma-\eta}}+\left[\xx(\omega)\right]^{\frac{1}{2(\gamma-\eta)}}_{2\gamma,\Delta_{[-\varepsilon,1-\varepsilon]}}\right)\\
			&+ 2^{4(1+\widetilde{N})}M^{1+\widetilde{N}}\widetilde{C}_1C_G\\
			&\times\sup_{\varepsilon\in [0,1]}\frac{1}{n}\sum_{j=1}^n \left(\left[X(\omega)\right]_{\gamma,[-\varepsilon,1-\varepsilon]}^{2}+\left[\xx(\omega)\right]^{2}_{2\gamma,\Delta_{[-\varepsilon,1-\varepsilon]}}+\left[X(\omega)\right]_{\gamma,[-\varepsilon,1-\varepsilon]}^{\frac{4(1+\widetilde{N})}{\gamma-\eta}}+\left[\xx(\omega)\right]^{\frac{2(1+\widetilde{N})}{\gamma-\eta}}_{2\gamma,\Delta_{[-\varepsilon,1-\varepsilon]}}\right).
		\end{align*}
		We take the limes superior in the previous expression. Then we obtain using the ergodic properties of the noise \eqref{ineq:ErgodicBound}
		\begin{align*}
			\limsup_{n\to \infty}\frac{1}{n}\log&\left(\sup_{\varepsilon\in [0,1]}\prod_{j=1}^n (1+H_1(\theta_{-j} \omega,[-\varepsilon,1-\varepsilon]))\right)\\
			&\leq \widetilde{M}(1+\widetilde{N})+\widetilde{M}(1+\widetilde{N})\chi^{\frac{1}{\gamma-\eta}}(K_{\frac{1}{\gamma-\eta}}+\mathbb{K}_{\frac{1}{2(\gamma-\eta)}})\\
			&+ 2^{4(1+\widetilde{N})}M^{1+\widetilde{N}}\widetilde{C}_1C_G\left(\mathbf{K}_2+K_{\frac{4(1+\widetilde{N})}{\gamma-\eta}}+\mathbb{K}_{\frac{2(1+\widetilde{N})}{\gamma-\eta}}\right)\leq c \mathbf{K}_q+c,
		\end{align*}
		with $c$ and $q$ defined in \eqref{eq:ConstantQandC}, where we used that $\chi^{\frac{1}{\gamma-\eta}}<1$. Therefore, the right-hand side of~\eqref{ineq:AssumptionOnLambda} depends only on $c$ and the moment of the Gaussian rough path $\X(\omega)$ of order $q$. Furthermore, there exists for any $\delta>0$ some $n_0\in \N_0$ such that for all $n\geq n_0$
		\begin{align}\label{eq:BoundProductH}
			\sup_{\varepsilon\in [0,1]}\prod_{j=1}^n (1+H_1(\theta_{-j} \omega,[-\varepsilon,1-\varepsilon]))\leq e^{(c  \mathbf{K}_q+c+\delta)n}.
		\end{align}
		Due to \eqref{ineq:AssumptionOnLambda} we have that $\lambda-c  \mathbf{K}_q-c-\delta>0$ for some small $\delta>0$. Then the temperedness of $y_0$, meaning that $e^{-\kappa t} \norm{y_0(\theta_t \omega)}_\alpha\to 0$ for $t\to \infty$ and $\kappa>0$, leads to
		\begin{align*}
			\norm{\phi(t,\theta_{-t}\omega,y_0(\theta_{-t}\omega))}_\alpha &\leq 1+P_1(\omega,[-1,1])\sup\limits_{\varepsilon\in [0,1]} \sum_{k=1}^{\infty}e^{-\lambda k} H_2(\theta_{-k}\omega,[-\varepsilon,1-\varepsilon])\\
			&\times \prod_{j=1}^{k-1} (1+H_1(\theta_{-j} \omega,[\varepsilon,1-\varepsilon]))+P_2(\omega,[-1,1]).
		\end{align*}
		It remains to show that the expression on the right-hand side is a tempered random variable. First, note that $P_2(\cdot,[-1,1])$ is integrable due to Lemma \ref{lem:IntegrableBound} and the fact that $\X$ is a Gaussian process. In particular, $\log(P_2(\cdot,[-1,1]))\in L^1(\Omega)$ and therefore
		\begin{align*}
			\limsup\limits_{t\to \infty} \frac{\log(P_2(\theta_{t}\omega,[-1,1]))}{t}=0,
		\end{align*}
		due to \cite[Theorem 4.1.3 i)]{Arnold}. So \eqref{eq:EquivTempered} is fulfilled, and $P_2(\cdot,[-1,1])$ is tempered. The same argument can be used to show a similar statement for $H_2(\cdot,[-\varepsilon,1-\varepsilon])$, so for every $\delta>0$ there exists a $n_0\in \N_0$ such that for $n\geq n_0$ and every $k\in \N$ we get the bound $$H_2(\theta_{-k}\omega,[-\varepsilon,1-\varepsilon])\leq e^{\delta n}.$$ 
		Together with \eqref{eq:BoundProductH} this leads to 
		\begin{align*}
			r(\omega):=\sup\limits_{\varepsilon\in [0,1]} &\sum_{k=1}^{\infty}e^{-\lambda k} H_2(\theta_{-k}\omega,[-\varepsilon,1-\varepsilon]) \prod_{j=1}^{k-1} (1+H_1(\theta_{-j} \omega,[-\varepsilon,1-\varepsilon]))\\
			&\leq \sum_{k=n_0}^\infty e^{-(\lambda-c  \mathbf{K}_q-c-\delta)n}+\sup\limits_{\varepsilon\in [0,1]} \sum_{k=1}^{n_0-1}e^{-\lambda k} H_2(\theta_{-k}\omega,[-\varepsilon,1-\varepsilon]) \\
			&\times \prod_{j=1}^{k-1} (1+H_1(\theta_{-j} \omega,[-\varepsilon,1-\varepsilon]))<\infty,
		\end{align*}
		due to \eqref{ineq:AssumptionOnLambda}, so the series is well-defined. In order to show the measurability, we recall that since $\X$ is a geometric rough path, the H\"older norms are continuous with respect to the time interval, according to Remark \ref{rem:RPContHolderNorm}. Therefore, $\varepsilon\mapsto H_i(\omega,[-\varepsilon,1-\varepsilon])$ for $i=1,2$ is continuous. In conclusion, the supremum can be taken over $[0,1]\cap \Q$ instead of $[0,1]$ entailing the measurability. Now, the temperedness of $r$ follows by similar arguments as in~\cite[Proposition 3.5]{DucHong23}. For this we need the continuity of the Hölder seminorms $[X]_{\gamma,[s,t]}$ and $[\xx]_{2\gamma,\Delta_{[0,T]}}$, compare Remark \ref{rem:RPContHolderNorm}, and that $H_i$ defined in \eqref{eq:DefH} satisfies $H_i(\theta_\tau \omega,[s,t])=H_i(\omega,[s+\tau,t+\tau])$ for $i=1,2$ and $\tau\in \R$. Further, we have
		\begin{align*}
			\limsup\limits_{\abs{t}\to \infty} \frac{\log(P_1(\theta_{t}\omega,[-1,1])r(\theta_{t}\omega))}{\abs{t}}\leq\limsup\limits_{\abs{t}\to \infty} \frac{\log(P_1(\theta_{t}\omega,[-1,1]))}{\abs{t}}+\limsup\limits_{\abs{t}\to \infty} \frac{\log(r(\theta_{t}\omega))}{\abs{t}}=0,
		\end{align*}
		which means that the temperedness of $R(\omega):=1+P_1(\omega,[-1,1])r(\omega)+P_2(\omega,[-1,1])$ follows from \cite[Theorem 4.1.3 i)]{Arnold} regarding the temperedness of $P_1 r(\cdot)$ and the one of $P_2$ showed above. \\
		In conclusion, $B(\omega):=B(0,R(\omega)+\overline{\delta})$, for some $\overline{\delta}>0$ is a random absorbing set for $\phi$ in $E_\alpha$. 
		\qed\\
	\end{proof}
	\begin{theorem}\label{thm:attractor}
		Under the Assumptions \ref{ass}, \ref{ass:Nonlin}, \ref{ass:Noise} and \eqref{ineq:AssumptionOnLambda} the random dynamical system $\phi$ associated to \eqref{eq:MainEq}, possesses a random pullback $\mathscr{D}$-attractor $\mathcal{A}(\omega)$. 
	\end{theorem}
	\begin{proof}
		Since Lemma \ref{lem:AbsorbingSet} ensures the existence of an absorbing set $B(\omega)\in \mathscr{D}$ in $E_\alpha$, we need a compactness argument such that Theorem \ref{thm:ExAttractor} provides a global attractor $\mathcal{A}(\omega)$. Therefore we  define
		\begin{align*}
			K(\omega):=\overline{\phi(T_1,\theta_{-T_1}\omega,B(\theta_{-T_1}\omega))}^{E_\alpha}\subset B(\omega),
		\end{align*}
		where $T_1:=\lceil T_B \rceil$ and $T_B$ is the absorbing time of $B(\omega)$. The fact that $K(\omega)$ is indeed absorbing is a direct consequence of the cocycle property and the fact that $B(\omega)$ is an absorbing set. The proof of the compactness is now based on the compact embedding $E_{\alpha+\beta}\hookrightarrow E_\alpha$ for $0<\beta<\min\{1-\sigma_F,\gamma-\sigma_G\}$. 
		Let $y_0(\omega)\in B(\omega)$
		and observe that
		\begin{align*}
			\norm{\phi(T_1,\theta_{- T_1}\omega, y_0(\theta_{-T_1}\omega))}_{\alpha+\beta}&\leq \norm{S_{T_1}y_0(\theta_{-T_1}\omega)}_{\alpha+\beta}\\
			&+\norm{\int_0^{T_1} S_{T_1-r}F(\phi(r,\theta_{- T_1}\omega, y_0(\theta_{-T_1}\omega)))~\txtd r}_{\alpha+\beta}\\
			&+\norm{\int_0^{T_1} S_{T_1-r}G(\phi(r,\theta_{- T_1}\omega, y_0(\theta_{-T_1}\omega)))~\txtd \X_r(\theta_{-T_1}\omega)}_{\alpha+\beta}.
		\end{align*}
		The first term can be bounded by using the fact that $y_0(\theta_{-T_1}\omega)\in B(\theta_{-T_1}\omega)=B(0,R(\theta_{-T_1}\omega)+\overline{\delta})$ and \eqref{hg:1} 
		\begin{align*} 
			\norm{S_{T_1}y_0(\theta_{-T_1}\omega)}_{\alpha+\beta}\leq C_{-\beta}e^{-\lambda_A T_1} T_1^{-\beta}\norm{y_0(\theta_{-T_1}\omega)}_\alpha\leq C_{-\beta}e^{-\lambda_A T_1} T_1^{-\beta} (R(\theta_{-T_1}\omega)+\overline{\delta})< \infty.
		\end{align*}
		For the other two terms, we use that $B(\omega)$ is an absorbing set with absorbing time $T_B\leq T_1\in \N$, therefore the drift term can be estimated as 
		\begin{align*}
			&\norm{\int_0^{T_1} S_{T_1-r}F(\phi(r,\theta_{- T_1}\omega, y_0(\theta_{-T_1}\omega)))~\txtd r}_{\alpha+\beta}\\&\leq C_{-\beta-\sigma_F}\int_0^{T_1} e^{-\lambda_A(T_1-r)}(T_1-r)^{-\beta-\sigma_F}\norm{F(\phi(r,\theta_{- T_1}\omega, y_0(\theta_{-T_1}\omega)))}_{\alpha-\sigma_F}~\txtd r\\
			&\leq C_FC_{-\beta-\sigma_F}\int_0^{T_1} e^{-\lambda_A(T_1-r)}(T_1-r)^{-\beta-\sigma_F}\norm{\phi(r,\theta_{- T_1}\omega, y_0(\theta_{-T_1}\omega))}_{\alpha}~\txtd r\\
			&\leq C_FC_{-\beta-\sigma_F}K(\omega)\int_0^{T_1} e^{-\lambda_A(T_1-r)}(T_1-r)^{-\beta-\sigma_F}~\txtd r<\infty,
		\end{align*}
		where the integral is finite due to $\beta+\sigma_F<1$. For the rough integral we combine \eqref{estimate:integral}, \eqref{ineq:Nonlin} and \eqref{ineq:VarRiedelEst} 
		\begin{align*}
			&\norm{\int_0^{T_1} S_{T_1-r}G(\phi(r,\theta_{- T_1}\omega, y_0(\theta_{-T_1}\omega)))~\txtd \X_r(\theta_{-T_1}\omega)}_{\alpha+\beta}\\
			&\leq C_I C_G \rho_{\gamma,[0,T_1]}(\X(\theta_{-T_1}\omega))^2 T_1^{-\beta+\gamma-\sigma_G} \\
			&\times(1+\norm{\phi(\cdot,\theta_{- T_1}\omega, y_0(\theta_{-T_1}\omega)),G(\phi(\cdot,\theta_{- T_1}\omega, y_0(\theta_{-T_1}\omega)))}_{\cD^{2\gamma}_{X,\alpha-\sigma_G}([0,T_1])})\\
			&\leq C_I C_G \rho_{\gamma,[0,T_1]}(\X(\theta_{-T_1}\omega))^2 T_1^{-\beta+\gamma-\sigma_G}\\
			&\times(1+\norm{\phi(T_1,\theta_{- T_1}\omega, y_0(\theta_{-T_1}\omega))}_\alpha) P_1(\theta_{-T_1}\omega,[T_1,T_1+1])+ P_2(\theta_{-T_1}\omega,[T_1,T_1+1])\\
			&\leq C_I C_G \rho_{\gamma,[0,T_1]}(\X(\theta_{-T_1}\omega))^2 T_1^{-\beta+\gamma-\sigma_G}(1+R(\omega))P_1(\theta_{-T_1}\omega,[T_1,T_1+1])+ P_2(\theta_{-T_1}\omega,[T_1,T_1+1])\\
			&< \infty,
		\end{align*}
		since $\beta+\sigma_G<\gamma$. This shows that
		\begin{align*}
			\norm{\phi(T_1,\theta_{- T_1}\omega, y_0(\theta_{-T_1}\omega))}_{\alpha+\beta}\leq \infty,
		\end{align*}
		for arbitrary $y_0(\omega)\in B(\omega)$, which leads to $K(\omega)\in E_{\alpha+\beta}$. Therefore $K(\omega)$ is a compact absorbing set in $E_{\alpha}$. 
		\qed
	\end{proof}
	\begin{remark}
		\begin{itemize}
			\item [i)]
			Condition~\eqref{ineq:AssumptionOnLambda} means that an attractor of \eqref{eq:MainEq} exists if the nonlinearities are sufficiently small in comparison to the spectral bound of the operator $A$. However this criterion is more flexible than the one derived in~\cite{LinYangZeng23}, where it is assumed that $\max\{L_G, \|G(0)\|_{\alpha-\sigma}, \|DG(0)G(0)\|_{\alpha-\gamma-\sigma}\}<1$. Here $L_G$ incorporates the Lipschitz constants of $G$, $DG$ and $D^2G$. In contrast,  our technique also offers the possibility of taking $C_G>1$ by selecting a sufficiently large $\lambda_A$. 
			A similar condition to~\eqref{ineq:AssumptionOnLambda} was obtained in \cite{DucHong23} for Young differential equations.
			We note that we have to control higher moments of the noise term $\X$, which is possible since $\mathbf{X}$ is assumed to be a Gaussian rough path. 
			\item[ii)] If the noise is additive, the existence of a random dynamical system can be established, transforming the SPDE into a random PDE using the stationary Ornstein-Uhlenbeck process. In this case, the condition for the existence of the attractor \eqref{ineq:AssumptionOnLambda} simplifies to
			\begin{align*}
				\lambda_A>2(C_{-\sigma_F} C_F\Gamma(1-\sigma_F))^{\frac{1}{1-\sigma_F}},
			\end{align*}
			which is consistent with other results for additive noise and the assumption $F:E_\alpha\to E_\alpha$ (consequently $\sigma_F=0)$, compare \cite[Assumption 2]{KuhnNeamtuSonner21}.
		\end{itemize}
	\end{remark}
	Going back to our setting, a major advantage of this method is that it directly allows us to investigate the regularity of the random attractor.
	
	\begin{corollary}\label{cor:reg} 
		Under the Assumptions \ref{ass}, \ref{ass:Nonlin} and \ref{ass:Noise} let $0<\beta<\min\{1-\sigma_F,\gamma-\sigma_G\}$ and  assume further 
		\begin{align*}
			\lambda_A-2(C_{-\sigma_F-\beta}C_F\Gamma(1-\sigma_F-\beta))^{\frac{1}{1-\sigma_F-\beta}}&>c_\beta (\mathbf{K}_q+1) ,
		\end{align*}
		with $\widetilde{C}_{1,\beta}=\max\{C_I,C_{-\beta}C_I\}e^{\lambda_A}\min\left\{\widetilde{L}, \frac{M_{1-\sigma_F-\beta}}{2}\right\}$ and 
		\begin{align*}
			c_\beta &:=C(\widetilde{N})\max\{\widetilde{M},\widetilde{C}_{1,\beta}C_G\}.
		\end{align*} Then the random pullback $\mathscr{D}$-attractor $\mathcal{A}$ obtained in Theorem \ref{thm:attractor} also belongs to $E_{\alpha+\beta}$. 
	\end{corollary}
	\begin{proof}
		Let $y_0(\omega)\in D(\omega)$ where $D(\omega)\subset \mathscr{D}$ is a tempered set in $E_\alpha$ as in Lemma~\ref{lem:AbsorbingSet}. We need to estimate $\phi(t,\theta_{-t}\omega,y_0(\theta_{-t}\omega))$ in the $E_{\alpha+\beta}$-norm, similar as in Lemma \ref{lem:AbsorbingSet}. The only estimate, that we need to change, are the ones in Lemma \ref{lem:discreteEst}. Based on the computations in Lemma \ref{lem:IntegralEstimates} we observe that for $0<\beta<\min\{1-\sigma_F,\gamma-\sigma_G\}$ we can improve the estimates for the integral terms, exploiting the smoothing property \eqref{hg:1}
		\begin{align}\label{eq:ImprovedEstimate}
			\begin{split}
				\norm{\int_0^t S_{t-r}G(y_r)~\txtd \X_r(\omega)}_{\alpha+\beta}&\leq C_GC_{1} \sum_{l=0}^{\lfloor t \rfloor}e^{-\lambda_A(t-l-1)}P_3(\omega,[l,l+1]),\\
				\norm{\int_0^t S_{t-r}F(y_r)~\txtd r}_{\alpha+\beta}&\leq C_{-\sigma_F-\beta} C_F\int_0^t e^{-\lambda_A(t-r)} (t-r)^{-\sigma_F-\beta}\norm{y_r}_{\alpha}~\txtd r+C_{2,\beta},
			\end{split}
		\end{align}
		for all $t\geq 0$ and $C_{2,\beta}:=C_{-\sigma_F-\beta}C_F \lambda_A^{\sigma_F+\beta-1} \Gamma(1-\sigma_F-\beta)$. Further, we get $\|S_t y_0(\omega)\|_{\alpha+\beta} \leq C_{-\beta}e^{\lambda_A t} t^{-\beta} \norm{y_0(\omega)}_\alpha$ and combining this with \eqref{eq:ImprovedEstimate} one obtains similar as in Lemma \ref{lem:AbsorbingSet} an absorbing set. The compactness is shown in the same way as in Theorem \ref{thm:attractor}.
		\qed\\
	\end{proof}

	\section{Applications}\label{sec:app}
	We provide examples for the nonlinear term $G$ and indicate how the condition~\eqref{ineq:AssumptionOnLambda} on the existence of random attractors can be verified in concrete applications. Since the conditions on $F$ are less restrictive than those on $G$, we can consider in both examples a global Lipschitz nonlinearity $F$. Due to the specific form of the condition \eqref{ineq:AssumptionOnLambda}, the resulting constant $C_F$ can be compensated by $A$ or $G$. Therefore, we focus on $G$ and $A$. We further recall that according to Remark \ref{rem:RelaxAssump} the Assumption \ref{ass:Nonlin} on $G$ can be weakened, therefore it is enough to verify the assumptions specified in Remark \ref{rem:RelaxAssump}.
	\subsection{PDEs with multiplicative rough boundary noise}
	As in~\cite{NS23}, we let $\cO\subset \R^d$ be a bounded domain with $C^\infty$-boundary and consider the semilinear parabolic evolution equation with multiplicative rough boundary noise in $E:=L^p(\cO)$, for $1<p<\infty$, given by
	\begin{align}\label{eq:BoundNoise}
		\begin{cases}
			\frac{\partial}{\partial t} y = \cA y & \text{ in } \cO,\\
			\cC y = G(y)~\frac{\txtd}{\txtd t} \X & \text{ on } \partial \cO,\\
			y(0)=y_0.
		\end{cases}
	\end{align}
	Similar problems were treated in~\cite{VeraarSchnaubelt}, where the boundary noise is a Brownian motion and in~\cite{DuncanMaslowski}, where  additive fractional noise was considered. 
	Here $\X$ is a $\gamma$-H\"older rough path cocycle which satisfies Assumption \ref{ass:Noise} with $\gamma\in(\frac{1}{3},\frac{1}{2}]$, for example, the rough path lift of a fractional Brownian motion. Furthermore, $\cA$ is a formal second order differential operator in divergence form with Neumann boundary conditions $\cC$ given by 
	\begin{align}\label{formalOp}
		\cA u:= \sum_{i,j=1}^d \partial_i \left(a_{ij}\partial_j \right)u -\lambda_A u,\quad \cC u:= \sum_{i,j=1}^d \nu_i\gamma_\partial a_{ij}\partial _ju,
	\end{align}
	where the coefficients $a_{ij}:\overline{\cO}\to \R$ are smooth, $\lambda_A>0$ a constant, $(a_{ij})_{i,j=1}^d$ is symmetric and uniform elliptic, meaning that there exists some constant $k>0$ such that for all $\xi \in \R^d$ and $x\in \overline{\cO}$ we have
	\begin{align*}
		\sum_{i,j=1}^d a_{ij}(x)\xi_i \xi_j\geq k \abs{\xi}^2.
	\end{align*}
	Let $A:D(A)\subset E \to E$ be the $E$-realization of $(\cA,\cC)$ with $D(A):=\{u\in H^{2,p}(\cO)~:~\cC u=0\}$ and $(E_\alpha)_{\alpha\in \R}$ the respective fractional power scale. Since we consider a boundary value problem, we need also a second scale for the boundary data $\widetilde{E}_\alpha:=H^{\alpha,p}(\partial \cO)$. 
	Further, we introduce the Neumann operator $N$, which is the solution operator to 
	\begin{align*}
		\cA u&=0\ \text{in}\ \cO,\\
		\cC u &=g \ \text{on}\ \partial\cO,
	\end{align*}
	and satisfies $N\in \cL(\widetilde{E}_\alpha;E_{\varepsilon})$ for some $\varepsilon<\frac{1}{2}+\frac{1}{2p}$. In \cite{NS23} it was shown one can transform \eqref{eq:BoundNoise} in a semilinear problem without boundary noise. This reads as
	
	\begin{align}\label{eq:EquivBoundNoise}
		\begin{cases}
			\txtd y = Ay~\txtd t + A_{-\theta-2\gamma}NG(y)~\txtd \X_t,\\
			y(0)=y_0\in E_{-\theta},
		\end{cases}
	\end{align}
	where $\theta:=1-\varepsilon$ and $A_{-\theta-2\gamma}\in \cL(E_{1-\theta-2\gamma};E_{-\theta-2\gamma})$ is an extrapolation operator, see \cite[Chapter V]{Amann95}. Note, that in \cite{NS23} the equivalence of~\eqref{eq:BoundNoise} and~\eqref{eq:EquivBoundNoise} was proved for $A_{-\theta-\gamma}$ instead of $A_{-\theta-2\gamma}$.
	Since $A_{-\theta-\gamma}\subset A_{-\theta-2\gamma}$,
	we work now with $ A_{-\theta-2\gamma}$. This is due to the fact that $A_{-\theta-\gamma}NG$ does not satisfy Assumption \ref{ass:Nonlin}, since $A_{-\theta-\gamma}NG(y)$ is not well-defined for $y\in E_{-\eta-2\gamma}$. 
	\begin{remark}
		For a better comprehension, we recall here that extrapolation operators are necessary since
		$Ny$ does not belong to $D(A)$. Due to this reason, we need an extension of $A$ called extrapolation operator $A_{-\iota}$, for a suitable choice of $\iota$. This means that  $A_{-\iota}Ny$ is well-defined.
	\end{remark}

	\begin{theorem}
		Let $\lambda_A>0$ be large enough such that \eqref{ineq:AssumptionOnLambda} holds. We further assume that there exists a $\sigma>\theta+1+\frac{1}{p}$ such that for any $\vartheta \in \{0,\gamma, 2\gamma \}$ the diffusion term $G:E_{-\theta-\vartheta}\to \widetilde{E}_{-\theta-\vartheta+\sigma}$ is three times continuously Fr\'{e}chet differentiable with bounded derivatives and that the derivative of 
		\begin{align*}
			DF(\cdot) \circ A_{-\sigma} N G(\cdot):\cB_{-\theta-\gamma}\to \widetilde{\cB}_{-\theta-\gamma+\sigma}
		\end{align*}
		is bounded, $A$ has a compact resolvent and the principal part, i.e. $\widetilde{A}:=\sum_{i,j=1}^d \partial_i \left(a_{ij}\partial_j \right)$ is dissipative and there exists a constant $a_0>0$ such that $\widetilde{A}-a_0$ is surjective. Then there exists a random dynamical system for \eqref{eq:EquivBoundNoise}on $E_{-\theta}$ which possesses a global random $\mathscr{D}$-attractor.
	\end{theorem}
	\begin{proof}
		As in \cite{NS23} it can be shown that $A_{-\theta-2\gamma}NG(y)$ satisfies the assumptions stated in Remark \ref{rem:RelaxAssump} with $\sigma_G=0$. Consequently, there exists a global-in-time solution which generates a random dynamical system \cite[Theorem 3.20, 4.3]{NS23} on $E_{-\theta}$.  
		
		It is known that the operator $A$ satisfies Assumption \ref{ass} i). Due to the maximal dissipativity assumption, Lumer-Phillips Theorem yields the existence of an analytic semigroup $(\widetilde{S}_t)_{t\geq 0}$ of contractions for the principal part $\widetilde{A}$. 
		Then the semigroup generated by $A$ is given by $S_t=e^{-\lambda_A t}\widetilde{S}_t$ and is therefore exponential stable with parameter $\lambda_A$, since $(\widetilde{S}_t)_{t\geq 0}$ is contractive. The compactness of the semigroup, follows from the fact that $A$ has a compact resolvent, \cite[V.1.2.1]{Amann95}. The existence of the global attractor is then ensured by Theorem \ref{thm:attractor}.
		\qed\\
	\end{proof}

	For the sake of completeness, we conclude with an example in the Young regime, i.e.~$\gamma\in (\frac{1}{2},1)$.
	
	\begin{example}
		If the noise $X\in C^{\widetilde{\gamma}}$ is Hölder continuous with parameter $\widetilde{\gamma}>\frac{1}{2}$, we can use the Young integral instead of the rough one. The Young integral can be defined for a path $y\in\mathfrak{C}:= C(\cB_\alpha)\cap C^{\widetilde{\gamma}}(\cB_{\alpha-\widetilde{\gamma}})$ as
		\begin{align*}
			\int_s^t S_{t-r} y_r~\txtd X_r =\lim\limits_{|\cP|\to 0} \sum\limits_{[u,v]\in\cP} S_{t-u} y_u X_{v,u}, 
		\end{align*}
		whereas~\eqref{estimate:integral} reads as
		\begin{align*}
			\Big\|\int\limits_{s}^{t} S_{t-r} y_{r}~\txtd X_{r} - S_{t-s}y_{s}X_{t,s}\Big\| _{\alpha-\widetilde{\gamma}+\beta} \lesssim \|y\|_{\mathfrak{C}} [X]_{\widetilde{\gamma}}  (t-s)^{2\widetilde{\gamma}-\beta},
		\end{align*}
		for $\beta<2\widetilde{\gamma}$. In order to solve~\eqref{eq:MainEq}, the nonlinear term $G$ needs to be only two times Fréchet differentiable. The
		existence thery of random attractors developed in this work carries over to the Young case as well, modifying~\eqref{ineq:AssumptionOnLambda} accordingly. More precisely, the constant $c$ differs and $\mathbf{K}_q = K_q$. 
		As a particular application of Theorem~\ref{thm:attractor} we obtain a random pullback $\mathscr{D}$-attractor for~\eqref{eq:BoundNoise} with Dirichlet boundary noise instead of Neumann. The global well-posedness of~\eqref{eq:BoundNoise} with Dirichlet boundary noise in the Young regime was established in \cite[Theorem 3.24]{NS23}.
	\end{example}

	\subsection{Parabolic PDEs with multiplicative rough noise} 
	\begin{example}\label{ex:Nonlin1}
		We consider the  parabolic PDE on $E:=L^p(\cO)$ for $1<p<\infty$ given by 
		\begin{align}\label{eq:ParabolEq}
			\begin{cases}
				\txtd y = A y~\txtd t + G(y)~\txtd \X_t, \\
				y(0)= y_0\in E_{\alpha},
			\end{cases}
		\end{align}
		where $A$ is a second order operator in $E$ which fulfills Assumption \ref{ass} with $\lambda_A>0$. For example, $A:=\Delta_D-\lambda_A$ where $\Delta_D$ denotes the Dirichlet-Laplacian and $\lambda_A>0$. The corresponding scale is  given by 
		\begin{align*}
			E_{\beta}:=
			\begin{cases}
				H^{2\beta,p}(\cO),     & 0\leq \beta<\frac{1}{2p} \\
				H^{2\beta,p}_0(\cO),    & \frac{1}{2p}<\beta\leq 1, \beta\neq \frac{p+1}{2p}.
			\end{cases}
		\end{align*}
		We define for some $\alpha\in \R$, $\sigma<\gamma$ and $\vartheta\in \{0,\gamma,2\gamma\}$ the mapping
		\begin{align}\label{linearG}
			G:E_{\alpha-\vartheta}\to E_{\alpha-\vartheta-\sigma},u\mapsto (-\Delta_D)^\sigma u.
		\end{align}
	\end{example}
	\begin{theorem}
		Let $\lambda_A$ be large enough such that \eqref{ineq:AssumptionOnLambda} is satisfied.~Then there exists a random dynamical system $\phi$ on $E_\alpha$ for \eqref{eq:ParabolEq} that possesses a  random pullback $\mathscr{D}$-attractor.
	\end{theorem}
	\begin{proof}
		In this case $G$ is a linear bounded operator, therefore the assumptions in Remark \ref{rem:RelaxAssump} are satisfied, compare \cite[Example 4.1]{HN22}. The existence of the random pullback attractor follows from Theorem \ref{thm:attractor}.
		\qed 
	\end{proof}
	\begin{example}\label{ex:Nonlin2}
		We consider a nonlinear integral operator $G$,  similar to \cite[Section 7]{HN19} and~\cite[Section 7]{GLS16}. Here we treat the evolution equation on $E:=L^p(\cO)$, for $p\geq 3$, given by
		\begin{align}\label{eq:ParabolEq2}
			\begin{cases}
				\txtd y = A y~\txtd t + G(y)~\txtd \X_t, \\
				y(0)= y_0\in E_{\alpha},
			\end{cases}
		\end{align}
		where $A=\Delta_D -\lambda_A$ with $\lambda_A>0$ as in the previous example. 
		We further choose $\alpha>3\gamma$ which leads to $E_{\alpha-\vartheta-\sigma}\hookrightarrow L^p(\cO)$. We introduce the operator $G$ as
		\begin{align}\label{integralNonlin}
			G:E_{\alpha-\vartheta}\to E_{\alpha-\vartheta-\sigma},u\mapsto  \int_{\cO} g(\cdot,u(x))~\txtd x.
		\end{align}
		for $\vartheta\in \{0,\gamma,2\gamma\}$. The kernel $g:\overline{\cO}\times \R\to \R$ is assumed to be three times continuously differentiable with bounded derivates such that $\sup\limits_{x\in \R}\norm{D_2^k g(\cdot,x)}_{\alpha-\vartheta-\sigma}\leq C<\infty$ for $k=0,1,2,3$ and $g|_{\partial \cO \times \R}=0$.
	\end{example}
	\begin{theorem}
		Let $\lambda_A$ be large enough such that \eqref{ineq:AssumptionOnLambda} is satisfied. Then there exists a random dynamical system $\phi$ for \eqref{eq:ParabolEq2} on $E_\alpha$ that possesses a random pullback $\mathscr{D}$-attractor. 
	\end{theorem}
	\begin{proof}
		It is easy to see, that $G$ is well-defined and similar as in \cite[XVII.3]{Kantorovich} three times Fréchet differentiable with
		\begin{align*}
			[DG(u)](h_1) &=\int_{\cO} D_2g(\cdot,u(x))h_1(x)~\txtd x,\\
			[D^2G(u)](h_1,h_2) &=\int_{\cO} D_2^2 g(\cdot,u(x))h_1(x)h_2(x)~\txtd x,\\
			[D^3G(u)](h_1,h_2,h_3) &=\int_{\cO} D_2^3 g(\cdot,u(x))h_1(x)h_2(x)h_3(x)~\txtd x.
		\end{align*}
		Due to the boundedness assumption on $g$ and regarding the embedding $E_{\alpha-\vartheta-\sigma}\hookrightarrow L^p(\cO)$, all these derivatives are bounded operators from $E^{\otimes k}_{\alpha-\vartheta}$ to $E_{-\vartheta}$ for $k=1,2,3$ and $\vartheta\in \{0,\gamma,2\gamma\}$. For the derivative of $DG(u)\circ G(u)$, we can similarly show that
		\begin{align*}
			DG(u)\circ G(u)=\int_{\cO}D_2g(\cdot,u(x))\int_{\cO} g(x,u(y))~\txtd y\txtd x
		\end{align*}
		maps from $E_{\alpha-\gamma}$ to $E_{\alpha-2\gamma-\sigma}$. The derivative of this operator is  given by
		\begin{align*}
			(D(DG(u)\circ G(u))](h)&=[D^2G(u)](G(u),h)+[DG(u)](DG(u)h)\\
			&=\int_{\cO}\int_{\cO} \left(D_2^2 g(\cdot,u(x))g(x,u(y))+D_2g(\cdot,u(x))D_2g(x,u(y))\right)h(y)~\txtd y\txtd x,
		\end{align*}
		for $h\in E_{\alpha-\gamma}$. Therefore the boundedness of the derivative can be showed using  the assumption on $g$. This means that the condition on $G$ imposed in~\cite{HN22} and mentioned in Remark \ref{rem:RelaxAssump} is satisfied. In conclusion the existence of a global attractor  for~\eqref{eq:ParabolEq2} follows from Theorem \ref{thm:attractor}.
		\qed
	\end{proof}
	\begin{remark}
		Note that in Example \ref{ex:Nonlin1} and \ref{ex:Nonlin2} we do not impose as  in~\cite{LinYangZeng23} that  $\max\{L_G, \|G(0)\|_{\alpha-\sigma}, \|DG(0)G(0)\|_{\alpha-\gamma-\sigma}\}<1$, where $L_G$ incorporates the Lipschitz constants of $G$, $DG$ and $D^2G$.
	\end{remark}

\end{document}